\documentclass [11pt,oneside]{amsart}

\reversemarginpar \textwidth=14cm \textheight=19cm

\usepackage{amssymb} \usepackage{amsfonts} \usepackage{amsmath}
\usepackage{amsthm} \usepackage{epsfig} 
\usepackage{tikz}
\usepackage[cmtip,arrow]{xy}
\usepackage{amsmath,amssymb,enumerate,pb-diagram,pb-xy}

\usepackage[applemac]{inputenc}
\input xy
\xyoption{all}

\usepackage{caption}

\newcommand{\tto}[1]{\stackrel{#1}{\longrightarrow}}
\newtheorem{lemma}{Lemma}[section]
\newtheorem{teo}[lemma]{Theorem}
\newtheorem{prop}[lemma]{Proposition}

\theoremstyle{definition}
\newtheorem{defn}[lemma]{Definition}

\newtheorem{rem}[lemma]{Remark}

\newcommand{\matB}{\mathbb{B}}

\newcommand{\wdtX}{\widetilde{X}}
\newcommand{\wdtW}{\widetilde{W}}

\newcommand{\matN}{\ensuremath {\mathbb{N}}}
\newcommand{\matR} {\ensuremath {\mathbb{R}}}

\newcommand{\calH} {\ensuremath {\mathcal{H}}}
\newcommand{\R} {\ensuremath {\mathbb{R}}}
\newcommand{\mea} {\mathrm{m}}
\newcommand{\mh} {\mathrm{mh}}

\newcommand{\hatgamma} {\widehat{\gamma}}
\newcommand{\lamtil} {\widetilde{\Lambda}}

\newcommand{\calC} {\ensuremath {\mathcal{C}}}

\newcommand{\str} {\ensuremath {{\rm str}}}
\newcommand{\strtil} {\ensuremath {\widetilde{\rm str}}}

\author{Roberto Frigerio}
\author{Cristina Pagliantini}

\address{Dipartimento di Matematica \\
Universit\`a di Pisa \\
Largo B.~Pontecorvo 5 \\
56127 Pisa, Italy}

\email{frigerio@dm.unipi.it, pagliantini@mail.dm.unipi.it}

\title[Measure homology and bounded cohomology of pairs]{Relative measure homology\\ and continuous
bounded cohomology\\ of topological pairs}



\thanks{}

\begin{document}
\begin{abstract}
Measure homology was introduced by Thurston 
in his notes about the geometry and topology of 3-manifolds, where it was
exploited in the computation of the simplicial volume of hyperbolic manifolds. 
Zastrow and Hansen independently proved 
that there exists a canonical isomorphism between measure homology and singular homology (on the category of CW-complexes), and it was then shown by 
L\"oh that, in the absolute case, 
such isomorphism is in fact an isometry
with respect to the $L^1$-seminorm on singular homology and the total variation seminorm 
on measure homology. L\"oh's result
plays a fundamental r\^ole in the use of measure homology as a tool for computing
the simplicial volume of Riemannian manifolds.

This paper deals with an extension of L\"oh's result to the relative case. 
We prove that relative 
singular homology and relative measure homology are isometrically isomorphic
for a wide class of topological pairs. Our results can be applied for instance in computing
the simplicial volume of Riemannian manifolds with boundary.

Our arguments are based on new results about continuous (bounded) cohomology
of topological pairs, which are probably of independent interest. 
\end{abstract}
\maketitle

\section{Introduction}
Measure homology was introduced by Thurston in~\cite{Thurston},
where it was  exploited in the proof that the 
simplicial volume of a closed hyperbolic $n$-manifold
is equal to its Riemannian volume divided by a constant only depending on
$n$ (this result is attributed in~\cite{Thurston} to Gromov).
In order to rely on measure homology, it is necessary to know that
this theory ``coincides''
with the usual real singular homology, at least for a large class of spaces.
The proof that measure homology and real singular homology 
of CW-pairs
are isomorphic has appeared in~\cite{Hansen,Zastrow}.
However, in order to exploit measure homology as a tool for computing
the simplicial volume, one has to show that these homology theories
are not only isomorphic, but also \emph{isometric} (with respect to
the seminorms introduced below). In the absolute
case, this result is achieved in~\cite{Loh}. Our paper is devoted
to extending L\"oh's result to the context of relative homology
of topological pairs. As mentioned in~\cite[Appendix A]{Fujiwara} 
and~\cite[Remark 4.22]{Lohtesi}, such an extension seems to rise difficulties
that suggest that L\"oh's argument should not admit a straightforward
translation
into the relative context. For a detailed account about the notion
of measure homology and its applications see \emph{e.g.}~the introductions
of \cite{Zastrow, Berlanga}.

In order to achieve our main results, we develop some aspects
of the theory of
continuous bounded cohomology of topological pairs. More precisely, we compare
such a theory with the usual bounded cohomology of pairs of
groups and spaces. In~\cite{Park},
Park provides the algebraic foundations to the theory of relative bounded cohomology,
extending Ivanov's homological algebra approach (see~\cite{Ivanov}) to the relative
case.  
However, 
Park endows the bounded cohomology of a pair of spaces with
a seminorm which is \emph{a priori}
different from the seminorm considered in this paper.
In fact,
the most common definition of simplicial volume 
is based on a specific
$L^1$-seminorm on singular homology, 
whose dual is just the $L^\infty$-seminorm on bounded cohomology
defined in~\cite[Section 4.1]{Gromov}. This seminorm does not coincide
\emph{a priori} with Park's seminorm, so our results cannot be deduced from
Park's arguments. More precisely, 
it is shown in~\cite[Theorem 4.6]{Park}
that Gromov's and Park's norms are  biLipschitz equivalent 
(see Theorem~\ref{biLipschitz} below). In~\cite[page 206]{Park}
it is stated that it remains unknown if this equivalence is actually an isometry. 
In Section~\ref{park:sec} we answer this question in the negative, providing examples
showing that Park's and Gromov's seminorms 
indeed do not coincide in general.

\subsection{Relative singular homology of pairs}
Let $X$ be a topological space, and
$W\subseteq X$ be a (possibly empty) subspace of $X$.
For $n\in\matN$ we denote  by $C_n(X)$ the module
of singular $n$-chains with real coefficients, \emph{i.e.}~the $\matR$-module
freely generated by the set $S_n (X)$ of singular 
$n$-simplices with values in $X$. 
The natural inclusion of $W$ in $X$ induces an inclusion of 
$C_n (W)$ into $C_n(X)$,
and we denote by 
$C_n (X,W)$ the quotient space $C_n (X)/C_n (W)$. 
The usual differential of the complex $C_\ast (X)$ defines
a differential $d_\ast\colon\thinspace 
C_{\ast} (X,W)\to C_{\ast -1} (X,W)$. The homology of the resulting
complex is the usual relative singular homology of the 
topological pair $(X,W)$, and will be denoted
by $H_\ast (X,W)$.

The $\matR$-vector space $C_n (X,W)$
can be endowed with the following natural $L^1$-norm:  
if $\alpha\in C_n (X,W)$, then
$$
\|\alpha\|_1 =\inf \left\{\sum_{\sigma\in S_n(X)} 
|a_\sigma|\, ,\, {\rm where}\ 
\alpha=\left[\sum_{\sigma\in S_n (X)} a_\sigma \sigma\right]
\ {\rm in}\ C_n(X)/C_n (W) \right\}.
$$ 
Such a norm
descends to a seminorm on $H_n (X,W)$, which is defined as follows:
if $[\alpha]\in H_n (X,W)$, then
$$
\|[\alpha ]\|_1  =  \inf \{\|\beta \|_1\;|\; \beta\in C_n (X,W),\, 
d_n\beta=0,\, [\beta ]=[\alpha ] \}
$$
(this seminorm can be null on non-zero elements of $H_n (X,W)$).
Of course, we may recover the absolute homology modules of $X$
just by setting $H_n(X)=H_n(X,\emptyset)$.

\subsection{Relative measure homology of pairs}
Let us now recall the definition of relative measure homology
of the pair $(X,W)$.
We endow $S_n(X)$ with the compact-open
topology and denote by $\matB_n(X)$ the $\sigma$-algebra
of Borel subsets of $S_n(X)$. 
If $\mu$ is a signed measure on $\matB_n(X)$ (in this case we say for short
that $\mu$ is a Borel measure on $S_n(X)$), the
\emph{total variation of $\mu$} is defined by the formula
$$\|\mu\|_\mea=\sup_{A\in \matB_n(X)}\mu(A)-\inf_{B\in\matB_n(X)}\mu(B)\quad \in\, [0,+\infty]$$
(the subscript $\mea$ stands for \emph{measure}).
For every $n\geq 0$, the measure chain module
$\mathcal{C}_n(X)$
is the real vector space of the Borel measures on $S_n(X)$
having finite total variation and admitting a compact determination set.
The graded module $\mathcal{C}_\ast (X)$ can be given the structure
of a complex via the boundary operator
$$
\begin{array}{rccc}
\partial_{n}\colon &\calC_{n}(X)&\longrightarrow &\calC_{n-1}(X)\\
&\mu &\longmapsto &\sum_{j=0}^{n}(-1)^j\mu^j\ , 
\end{array}
$$
where $\mu^j$
is the push-forward of $\mu$ 
under the map that takes a simplex $\sigma\in S_n(X)$ into the composition of 
$\sigma$ with the usual inclusion
of the standard $(n-1)$-simplex onto the $j$-th face of $\sigma$.

Let now $W$ be a (possibly empty) subspace of $X$. 
It is proved in~\cite[Proposition 1.10]{Zastrow} that the $\sigma$-algebra $\matB_n(W)$
of Borel subsets of $S_n(W)$ coincides with the set
$\{A\cap S_n(W)\;|\;A\in \matB_n(X)\}$.
For every
$\mu\in \mathcal{C}_n(W)$, the assignement
$$\nu(A)=\mu(A\cap S_n(W)),\qquad A \in \matB_n(X)\ ,
$$
defines a Borel measure on $S_n (X)$, which is called
the \emph{extension} of $\mu$.
If $\mu$ has compact determination set and
finite total variation then the same is true for $\nu$,
so that we have a natural inclusion $\calC_n(W)\hookrightarrow \calC_n(X)$
(see \cite[Proposition 1.10 and Lemma 1.11]{Zastrow} for full details). 
The image of $\calC_n(W)$ in $\calC_n(X)$ will be simply denoted by $\calC_n(W)$,
and coincides with the set of the elements of $\calC_n(X)$ which admit a 
compact determination
set contained in $\matB_n(W)$. 
We denote by $\calC_n(X,W)$
the quotient $\calC_n(X)/\calC_n(W)$. 

It is readily seen that 
$\partial_n(\calC_n(W))\subseteq \calC_{n-1}(W)$, so
$\partial_n$ induces 
a boundary operator $\calC_n(X,W)\to\calC_{n-1}(X,W)$, which will still 
be denoted by $\partial_n$.  
The homology of the complex $(\calC_\ast(X,W),\partial_\ast)$ is the \emph{relative
measure homology of the pair $(X,W)$}, 
and it is denoted by $\mathcal{H}_\ast(X,W)$.

Just as in the case of singular homology, we may endow $\calH_n(X,W)$ with a seminorm as
follows.
For every $\alpha\in \calC_n(X,W)$ we set 
$$
\|\alpha\|_\mea =\inf \left\{\|\mu\|_\mea ,\, {\rm where}\ \mu\in\calC_n (X),\
[\mu]=\alpha
\ {\rm in}\ \calC_n(X,W)=\calC_n(X)/\calC_n (W) \right\}.
$$ 
Then, for every $[\alpha]\in \calH_n(X,W)$ we set
$$\|[\alpha]\|_\mh =\inf\{\|\beta\|_\mea \;|\;\beta\in\calC_n(X,W),\;\partial_n\beta=0,\;[\beta]=[\alpha]\}$$
(the subscript $\mh$ stands for \emph{measure homology}).
The absolute measure homology module $\calH_n(X)$ can be defined
just by setting $\calH_n(X)=\calH_n(X,\emptyset)$.


\subsection{Relative singular homology v.s.~relative measure homology}
For every $\sigma\in S_n(X)$ let us denote by $\delta_\sigma$ the atomic
measure supported by the singleton $\{\sigma\}\subseteq S_n(X)$. The chain
map
$$
\begin{array}{rccl}
\iota_\ast\colon &C_\ast(X,W)&\longrightarrow& \calC_\ast(X,W)\\
& \sum_{i=0}^ka_i\sigma_i&\longmapsto &\sum_{i=0}^k a_i\delta_{\sigma_i},
\end{array}
$$
induces a map
$$H_n(\iota_\ast)\colon H_n(X,W)\longrightarrow \calH_n(X,W),\quad n\in\mathbb{N}\ ,
$$
which is obviously norm non-increasing for every $n\in\mathbb{N}$.

The following result is proved in~\cite{Zastrow, Hansen}:

\begin{teo}[\cite{Zastrow, Hansen}]\label{teoza}
Let $(X,W)$ be a CW-pair.
For every $n\in\mathbb{N}$, the map
$$
H_n(\iota_\ast)\colon H_n(X,W)\longrightarrow \calH_n(X,W)
$$
is an isomorphism.
\end{teo}

Zastrow's and Hansen's proofs of Theorem~\ref{teoza} are based on the
fact that 
relative measure homology satisfies the Eilenberg-Steenrod axioms
for homology (on suitable categories of topological pairs). Therefore,
their approach avoids the explicit construction of the inverse maps
$H_n(\iota_\ast)^{-1}$, $n\in\mathbb{N}$, and does not give 
much information about the behaviour of such inverse maps
with respect to the seminorms introduced above. 
In the case when $W=\emptyset$, the fact that $H_n(\iota_\ast)$ is indeed an isometry was
proved by  L\"oh:

\begin{teo}[\cite{Loh}]\label{loh:teo}
 If $X$ is any connected CW-complex, then for every $n\in\matN$ the map
$$
H_n(\iota_\ast)\colon H_n(X)\to\calH_n(X)
$$
is an isometric isomorphism.
\end{teo}

L\"oh's proof of Theorem~\ref{loh:teo}
exploits some deep results about the \emph{bounded cohomology}
of groups and topological spaces. In Sections~\ref{boundedco:sec}, \ref{boundedco2:sec} we develop
a suitable relative version of such results, which are exploited
in Subsection~\ref{teo1proof:sub} for proving the following:

\begin{teo}\label{teo1}
Let $(X,W)$ be a CW-pair, and let us suppose that the following conditions hold:
\begin{enumerate}
\item
$X$ is countable, and both $X$ and $W$ are connected;
\item
the map $\pi_j (W)\to \pi_j (X)$ induced by
the inclusion $W\hookrightarrow X$ is injective for $j=1$, and it is an isomorphism
for $j\geq 2$.
\end{enumerate}
 Then, for every $n\in\matN$ the isomorphism
$$
H_n(\iota_*)\colon H_n(X,W)\to\calH_n(X,W)
$$
is isometric.
\end{teo}

In fact, we will deduce Theorem~\ref{teo1} from Theorem~\ref{teo2}
below concerning the relationships between continuous (bounded) cohomology
and singular (bounded) cohomology of topological pairs.

\begin{defn}
A CW-pair $(X,W)$ is \emph{good} if it satisfies conditions~(1) and~(2) of the statement of Theorem~\ref{teo1}.
\end{defn}

We conjecture that Theorem~\ref{teo1} holds even without the hypothesis
that the pair $(X,W)$ is good, so a brief comment about the places where this assumption
comes into play is in order. The fact that $W$ is connected and $\pi_1$-injective
in $X$ allows us to exploit
results regarding the bounded cohomology of a pair $(G,A)$, where $G$ is a group
and $A$ is a subgroup of $G$. In order to deal with the case when $W$ is \emph{not} assumed to be $\pi_1$-injective, one could probably  
build on results regarding the bounded cohomology of a pair $(G,A)$, where $A,G$ are groups
and 
$\varphi\colon A\to G$ is a homomorphism of $A$ into $G$. 
This case is treated \emph{e.g.}~in~\cite{Park} by means of a mapping cone construction. 
However, the mapping
cone introduced in~\cite{Park} does not admit a norm inducing Gromov's seminorm
in bounded cohomology, so Park's approach seems to be of no help to our purposes.
Perhaps it is easier to drop from the hypotheses of Theorem~\ref{teo1}
the requirement that $W$ be connected (provided that we still assume that every component
of $W$ is $\pi_1$-injective in $X$). Several arguments in our proofs
make use of cone constructions which are based on the choice of a basepoint
in the universal coverings $\wdtX$, $\wdtW$ of $X$, $W$. When $W$ is connected (and $\pi_1$-injective in $X$),  
the space $\wdtW$ is realized as a connected subset of $\wdtX$, and this allows us to
define compatible 
cone constructions on $\wdtX$ and $\wdtW$. It is not clear
how to replace these constructions
when $W$ is disconnected:
one could
probably build on the theory of homology and cohomology of a group
with respect to any system of its subgroups, as described for instance
in~\cite{Bieri} (see also~\cite{Yaman}), but several difficulties arise
which we have not been able to overcome.
Finally, the assumption that $\pi_i(W)$ is isomorphic to $\pi_i(X)$ for every $i\geq 2$ plays a fundamental r\^ole in our proof of Proposition~\ref{standardok} below. 
On could get rid of this assumption by
using a result
stated without proof in~\cite[Lemma 4.2]{Park}, but at the moment we are not
able to provide a proof for Park's statement
(see Remark~\ref{Park:rem} for a brief discussion of this issue).

\subsection{Locally convex pairs}\label{LCP}
We are able to prove that measure homology is isometric to singular homology
also for a large family of pairs of metric spaces, 
namely for those pairs which support a \emph{relative straightening} for simplices.

The \emph{straightening procedure} for simplices was introduced by Thurston in~\cite{Thurston},
and establishes an isometric isomorphism between the usual singular homology of a space
and the homology of the complex of \emph{straight} chains. Such a procedure was originally defined on hyperbolic manifolds, and has then been extended to the context of 
non-positively curved Riemannian manifolds. 
In Section~\ref{convex:sec} we give the precise definition of \emph{locally convex pair
of metric spaces}.
Then, following some ideas described in~\cite{Loh-Sauer}, for every 
locally convex pair $(X,W)$ we define a 
straightening procedure
which induces a chain map between 
relative measure chains and relative singular chains. It turns out that
such a straightening 
induces a well-defined norm non-increasing map 
$\calH_n(X,W)\to H_n(X,W)$.
This map provides the desired norm non-increasing inverse of 
$H_n(\iota_\ast)$, so that we can prove (in Subsection~\ref{proofconvex:subsec})
the following:

\begin{teo}\label{teo3}
Let $(X,W)$ be a locally convex pair of metric spaces.
Then the map
$$H_n(\iota_\ast)\colon H_n(X,W)\longrightarrow \calH_n(X,W)
$$
is an isometric isomorphism for every $n\in\matN$.
\end{teo}

The class of locally convex pairs is indeed quite large, including
for example all the pairs $(M,\partial M)$, where $M$ is a non-positively curved
complete Riemannian manifold with geodesic boundary $\partial M$.

\begin{rem}
Suppose that $(X,W)$ is a locally convex pair, and 
let $K$ be a connected component of $W$. An easy application
of a metric version of
Cartan-Hadamard Theorem (see \emph{e.g.}~\cite[II.4.1]{Bridson})
shows that $\pi_1(K)$ injects into $\pi_1(X)$, and $\pi_i(K)=\pi_i(X)=0$
for every $i\geq 2$. In particular, if
$(X,W)$ is also a countable CW-pair and $W$ is connected, then $(X,W)$ is good,
and the conclusion of Theorem~\ref{teo3} also descends from Theorem~\ref{teo1}.
Note however that the request that $W$ be connected could be quite restrictive 
in several applications of our results. For example, 
it is well-known that the 
natural compactification 
of a complete finite-volume hyperbolic manifold with
geodesic boundary and/or cusps is a manifold with boundary $N$
admitting a locally CAT(0) (whence locally convex) metric that turns
the pair $(N,\partial N)$ into a locally convex pair (see \emph{e.g.}~\cite[pages 362-366]{Bridson}).
We have discussed in~\cite{FriPag} some properties of the simplicial volume of
such manifolds, and in that context several interesting examples
have in fact disconnected boundary. 
In~\cite{Pagliantini} it is shown how to apply Theorem~\ref{teo3}
for getting
shorter proofs of the main results of~\cite{FriPag}.
\end{rem}

\subsection{(Continuous) relative bounded cohomology}
As mentioned above, our proof of Theorem~\ref{teo1} 
involves the study of 
the relative bounded cohomology of topological pairs. Introduced by Gromov in~\cite{Gromov},
the relative bounded cohomology of pairs (of groups or spaces) seems
to be less clearly understood than absolute bounded cohomology. 
Here below we define the \emph{continuous}
(bounded) cohomology of topological pairs, and we put on (continuous) bounded cohomology
Gromov's $L^\infty$-seminorm which is ``dual'' (in a sense to be specified below) to
the seminorm on (measure) homology described above. 
Then, in Section~\ref{boundedco2:sec} we 
compare the continuous bounded cohomology of a  
good CW-pair to its usual singular bounded cohomology
(see Theorem~\ref{teo2} below). In Section~\ref{duality:sec} we show how this result implies
Theorem~\ref{teo1}. 

Let us now state more precisely our results. 
For every $n\in \matN$ we denote by $C^n(X)$ (resp.~$C^n(X,W)$)
the module of singular $n$-cochains with real coefficients, 
\emph{i.e.}~the algebraic dual of $C_n(X)$ (resp.~of $C_n(X,W)$).
We will often identify $C^n(X,W)$ with a submodule of $C^n(X)$ via the 
canonical isomorphism
$$C^n(X,W)\cong\{f\in C^n(X)\;|\;{f}|_{C_n(W)=0}\}\ .$$
If $\delta^\ast\colon C^\ast(X,W)\rightarrow C^{\ast+1}(X,W)$ is the usual differential,
the homology of the complex $(C^\ast(X,W),\delta^\ast)$ is the relative
singular cohomology of the pair $(X,W)$, and it is denoted by $H^\ast(X,W)$.

We regard $S_n(X)$ as a subset of $C_n(X)$, so that for every cochain
$\varphi\in C^n(X,W)\subseteq C^n(X)$ it makes sense to consider
the restriction $\varphi|_{S_n(X)}$. In particular, we say that
$\varphi$ is \emph{continuous} if $\varphi|_{S_n(X)}$ is (recall
that $S_n(X)$ is endowed with the compact-open topology).
 If we set 
$$
C_c^\ast (X,W)=\{\varphi\in C^\ast(X,W)\, |\, \varphi\ {\rm is\ continuous} \}\ ,
$$
then it is readily seen that $\delta^n (C_c^n(X,W))\subseteq C_c^{n+1}(X,W)$,
so $C_c^\ast (X,W)$ is a subcomplex of $C^\ast (X,W)$, whose homology
is denoted by $H_c^\ast (X,W)$.

Let us now come to the definition of (continuous) bounded cohomology.
We endow $C^n(X,W)$ with 
the $L^\infty$-norm defined by
$$\|f\|_\infty=\sup_{\sigma\in S_n(X)}|f(\sigma)|\ \in [0,\infty],\qquad f\in C^n(X,W)\ .$$

Let us introduce the following submodules of $C^\ast (X,W)$:
$$C^\ast_b(X,W)=\{f\in C^\ast(X,W)\;|\;\|f\|_{\infty}<\infty\}\ ,$$
$$C^\ast_{cb}(X,W)=C^\ast_b(X,W)\cap C^\ast_c(X,W)\ .$$

The coboundary map $\delta^n$ is bounded, so $C_b^\ast(X,W)$ (resp.~$C_{cb}^\ast (X,W)$) is a 
subcomplex of $C^\ast(X,W)$ (resp.~of $C_{c}^\ast (X,W)$).
Its homology is denoted by $H_b^\ast(X,W)$ (resp.~$H_{cb}^\ast(X,W)$),
and it is called 
the \emph{bounded cohomology} (resp.~\emph{continuous bounded cohomology}) of $(X,W)$.
The $L^\infty$-norm on $C^\ast (X,W)$ descends (after suitable restrictions)
to a seminorm on each of the modules $H^\ast (X,W)$,
$H_c^\ast (X,W)$, $H_{b}^\ast(X,W)$, $H_{cb}^\ast(X,W)$.
These seminorms will still 
be denoted by $\|\cdot \|_\infty$.
The inclusion maps 
$$   \rho_b^*\colon C_{cb}^*(X,W)\hookrightarrow C_b^*(X,W), \qquad 
\rho^*\colon C_c^*(X,W)\hookrightarrow C^*(X,W)
$$
induce maps 
$$H^*(\rho_b^\ast)\colon H_{cb}^*(X,W)\longrightarrow H_{b}^*(X,W)\ ,\qquad
H^*(\rho^\ast)\colon H^*_c(X,W)\longrightarrow H^*(X,W)\ , $$ 
that are a priori neither injective nor surjective.

We are now ready to state our main result about (continuous) bounded cohomology
of pairs, which is proved in Subsection~\ref{mainproof:sub}:
\begin{teo}\label{teo2}
Let $(X,W)$ be a good CW-pair. 
Then the map
$$H^n(\rho_b^\ast)\colon H_{cb}^n(X,W)\longrightarrow H_{b}^n(X,W)$$
admits a right inverse which is an isometric embedding (in particular,
$H^*(\rho_b^\ast)$ is surjective) for every $n \in \matN$. 
\end{teo}

In the absolute case, \emph{i.e.}~when $W=\emptyset$, Theorem~\ref{teo2}
is proved in~\cite[Theorem 1.2]{Frigerio}.
As observed at the end of Subsection~\ref{isom:sub},
the arguments developed in Section~\ref{boundedco2:sec}
also imply the following result
(see Section~\ref{boundedco:sec} for the definition of bounded cohomology
of pairs of groups):

\begin{teo}\label{gruppiespazi}
Let $(X,W)$ be a CW-pair. Then for every $n\in\matN$ there exists
an isomorphism between 
$H_b^n(\pi_1(X),\pi_1(W))$ and $H^n_{b}(X,W)$.
If in addition the pair $(X,W)$ is good, then
the isomorphism is isometric. 
\end{teo}

In Subsection~\ref{unbounded:sub} we show how 
Theorem~\ref{teo2} and \cite[Theorem 1.1]{Frigerio} 
can be exploited to prove the following:

\begin{teo}\label{teo2bis}
Let $(X,W)$ be a locally finite good CW-pair.
Then the map
$$H^n(\rho^\ast)\colon H_{c}^n(X,W)\longrightarrow H^n(X,W)$$
is an isometric isomorphism for every $n\in\matN$. 
\end{teo}

\subsection{Acknowledgements}
The authors thank Maria Beatrice Pozzetti for several useful conversations
about the contents of Ivanov's paper~\cite{Ivanov}.

\section{The case of locally convex pairs}\label{convex:sec}
The following definitions can be found for instance in~\cite{Bridson}.
Let $(X,d)$ be a metric space (when $d$ is fixed, we denote $(X,d)$ simply by $X$). 
A \emph{geodesic segment} in $X$ is an isometric embedding of 
a bounded closed interval into $X$.
The metric $d$ (or the metric space $X=(X,d)$) is 
\emph{geodesic} if every two points in $X$ are joined by a 
geodesic segment (in particular, $X$ is path connected and locally path connected).
Moreover, $d$ (or $X=(X,d)$) is
\emph{globally convex} if it is geodesic and if any two geodesic segments $c_1\colon [0,a]\to X$, $c_2\colon [0,a]\to X$
such that $c_1(0)=c_2(0)$ satisfy the condition $d(c_1(t),c_2(t))\leq td(c_1(a),c_2(a))$
for every $t\in [0,a]$ (and in this case, $X$ is contractible, see Lemma~\ref{contgeo} below). 
We say that $d$ (or $X=(X,d)$) is \emph{locally convex} if every point
in $X$ has a neighbourhood in which the restriction of $d$ is convex
(in particular, it is geodesic). 
A subspace $Y\subseteq X$ is \emph{convex} if every geodesic segment (in $X$)
joining any two points of $Y$ is entirely contained in $Y$ (in particular, $Y$ is path connected).

Let us suppose that $X$ is complete and locally convex. Then it is locally contractible,
hence it
admits a universal covering $p\colon\widetilde{X}\to X$. We endow $\widetilde{X}$ 
with the length metric induced by $p$, \emph{i.e.}~the
unique length metric $\widetilde{d}$ such that $p\colon (\widetilde{X},\widetilde{d})\to
(X,d)$ is a local isometry (see~\cite[Proposition I.3.25]{Bridson}). Since $(X,d)$
is complete and geodesic, the same is true for $(\widetilde{X},\widetilde{d})$.
Moreover, Cartan-Hadamard Theorem for metric spaces (see~\cite[II.4.1]{Bridson}),
implies that the space $(\widetilde{X},\widetilde{d})$ is globally convex.
 
Let $W$ be any subset of $X$. We say that $(X,W)$ is a 
\emph{locally convex pair of metric spaces} (or simply a \emph{locally convex pair})
if the following conditions hold:
\begin{enumerate}
 \item 
$X$ is complete and locally convex;
\item
$W$ is locally path connected;
\item
every path-connected component of $p^{-1}(W)\subseteq \widetilde{X}$
is convex in $\wdtX$. 
\end{enumerate}

Throughout the whole section we denote by $(X,W)$ a locally convex pair of 
metric spaces, we fix a universal covering 
$p\colon\widetilde{X}\to X$ (where $\wdtX$ is endowed with the induced
metric), and we 
denote by $\widetilde{W}$ the subset 
$p^{-1}(W)\subseteq \wdtX$ (on the contrary, in Section~\ref{boundedco2:sec} 
we will denote by $\widetilde{W}$ a fixed connected component
of $p^{-1}(W)$).

\subsection{Straight simplices}\label{straight1:sub}
In order to properly define straight simplices we first need the following
result, which is an immediate consequence of Cartan-Hadamard Theorem for
metric spaces:

\begin{lemma}[\cite{Bridson}, II.4.1]\label{contgeo}
For every pair of points $p,q\in
\widetilde{X}$ there exists a unique geodesic segment in $\widetilde{X}$
joining $p$ to $q$. 
Moreover, if $\alpha_{p,q}\colon [0,1]\to \widetilde{X}$ is a
constant-speed parameterization of such a segment, then 
$\alpha_{p,q}$ continuously depends (with respect to the compact-open topology) 
on $p$ and $q$. In particular, $\wdtX$ is contractible.
\end{lemma}

For $i\in\matN$ we denote by $e_i$ the point $(0,0,\ldots,1,\ldots,0,0,\ldots)\in \R^{\matN}$
where the unique non-zero coefficient is at the $i$-th entry (entries are indexed by $\matN$,
so $(1,0,\ldots)=e_0$). 
We denote by $\Delta_p$ the standard $p$-simplex, 
\emph{i.e.}~the convex hull of $e_0,\ldots,e_p$, and we observe
that with these notations we have $\Delta_p\subseteq \Delta_{p+1}$.

Let $k\in\matN$, and let $x_0,\ldots,x_k$ be points in $\widetilde{X}$. We recall here
the well-known definition of \emph{straight} simplex $[x_0,\ldots, x_k]\in S_k (\widetilde{X})$
with vertices $x_0,\ldots,x_k$:
if $k=0$, then $[x_0]$ is the $0$-simplex with image $x_0$; if straight simplices have
been defined for every $h\leq k$, then $[x_0,\ldots,x_{k+1}]\colon \Delta_{k+1}\to\widetilde{X}$
is determined by the following condition:
for every $z\in \Delta_k\subseteq \Delta_{k+1}$, the restriction
of $[x_0,\ldots,x_{k+1}]$ to the segment with endpoints
$z,e_{k+1}$ is a constant speed parameterization of
the geodesic joining $[x_0,\ldots,x_k] (z)$ to $x_{k+1}$
(the fact that $[x_0,\ldots,x_{k+1}]$ is well-defined and continuous is an immediate
consequence of Lemma~\ref{contgeo}).

\subsection{Nets}
Let 
$\Gamma\cong \pi_1 (X)$ be the 
group of the covering automorphisms of $p\colon \widetilde{X}\to X$, and observe
that, since $p$ is a local isometry, every element of $\Gamma$ is an isometry
of $\widetilde{X}$.

\begin{defn}\label{net:def}
A \emph{net} in $\widetilde{X}$ is given by a 
subset $\lamtil\subseteq \widetilde{X}$ and a locally finite collection
of Borel sets $\{\widetilde{B}_x\}_{x\in \lamtil}$ such that the following conditions hold:
\begin{enumerate}
\item
$\widetilde{X}=\bigcup_{x\in\lamtil} \widetilde{B}_x$ 
and $\widetilde{B}_x\cap \widetilde{B}_y=\emptyset$ for every $x,y\in \lamtil$
with $x\neq y$;
\item
$\gamma (\lamtil)=\lamtil$ for every $\gamma\in \Gamma$
and 
$\gamma (\widetilde{B}_x)=\widetilde{B}_{\gamma(x)}$ for every $x\in\lamtil$, $\gamma\in \Gamma$;
\item
if $\widetilde{K}$ is a path connected component of $\widetilde{W}$,
then
$\widetilde{K}\subseteq \bigcup_{x\in\lamtil\cap \widetilde{K}} \widetilde{B}_x$.
\end{enumerate}
\end{defn}


\begin{lemma}\label{net:lemma}
There exists a net.
\end{lemma}
\begin{proof}
For every $q\in X$ let us denote by $U_q$ an evenly-covered 
neighbourhood
of $q$ in $X$ (with respect to the universal covering $\wdtX\to X$).
Since $W$ is locally path connected,
we may also suppose that $W\cap U_q$ is path connected.
Being metrizable, $X$ is paracompact, so
the open covering $\{U_q\}_{q\in X}$ admits a locally finite open refinement
$\{V_i\}_{i\in I}$. Let us now fix a total ordering $\preceq$ on $I$ in such a way that
$i\preceq j$ whenever $V_i\cap W\neq \emptyset$ and $V_j\cap W=\emptyset$, and let us set
$$
B_i=V_i\setminus\left(\bigcup_{j\prec i} V_j\right)\ .
$$
By construction, the family $\{B_i\}_{i\in I}$ is locally finite in $X$. Moreover,
every $B_i$ is
the intersection of an open set and a closed set, therefore it is a Borel
subset of $X$. For every $i\in I$ let us choose $x_i\in B_i$ in such a way
that $x_i\in W$ whenever $B_i\cap W\neq \emptyset$, and let us set
$\Lambda=\bigcup_{i\in I} \{x_i\}$. We also set $B_{x_i}=B_i$ for every $i\in I$.

Let us now define $\lamtil=p^{-1}(\Lambda)$. 
For every $i\in I$ we choose an element
$\widetilde{x}_i\in p^{-1}(x_i)$, 
and we take $q_i\in X$ in such a way that $B_{x_i}\subseteq U_{q_i}$.
Being simply connected, $U_{q_i}$ lifts to the disjoint union
$p^{-1} (U_{q_i})=\bigcup_{\gamma\in\Gamma} \gamma(\widetilde{U}_{q_i})$,
where $\widetilde{U}_{q_i}$ is the connected component of $p^{-1} (U_{q_i})$
containing $\widetilde{x}_i$. 

We are now ready to define $\widetilde{B}_x$, where $x$ is any element
of $\lamtil$. In fact, every $x\in \lamtil$ uniquely determines
an index $i\in I$ and an element $\gamma\in \Gamma$ such that
$x=\gamma(\widetilde{x}_i)$, and we can set
$\widetilde{B}_x=\gamma(\widetilde{U}_{q_i}\cap p^{-1}(B_{x_i}))$. 
Of course $\widetilde{B}_x$ is a Borel subset of $\widetilde{X}$.

It is now easy to check that the pair $\left(\lamtil, \{\widetilde{B}_x\}_{x\in\lamtil}\right)$
provides a net: the local finiteness of the family
$\{\widetilde{B}_x,\, x\in\lamtil\}$
readily descends from the fact $p$ is a covering and $\{B_x,\, x\in\Lambda\}$ is locally finite
in $X$, and conditions~(1) and (2) of Definition~\ref{net:def}
are an obvious consequence of our choices. Let us now show that
condition~(3) also holds. We fix $x\in \widetilde{\Lambda}$ such that 
$\wdtW\cap \widetilde{B}_x\neq\emptyset$. By construction we have $x\in\wdtW$,
and there exist $\gamma\in\Gamma$ and $i\in I$ such
that $\widetilde{B}_x\subseteq \gamma(\widetilde{U}_{q_i})$. Our assumption
that $U_q\cap W$ is path connected implies that $\gamma(\widetilde{U}_{q_i})\cap
\wdtW$ is also path connected, so the 
set $\widetilde{B}_x\cap \wdtW$ is entirely contained
in the path connected component of $\wdtW$ containing $x$,
whence the conclusion.
\end{proof}

\subsection{Straightening}\label{straight2:sub}
We are now ready to define our straightening operator.
Let $\left(\lamtil, \{\widetilde{B}_x\}_{x\in\lamtil}\right)$ be a net.
We denote by
$S_n^{\lamtil} (\widetilde{X})\subseteq S_n (\widetilde{X})$
the set of straight $n$-simplices
in $\widetilde{X}$ with vertices in $\lamtil$. Then we let 
$\strtil_n \colon C_n (\widetilde{X})\to
C_n (\widetilde{X})$ be the unique linear map such that 
for $\widetilde{\sigma}\in S_n (\widetilde{X})$
$$
\strtil_n (\widetilde{\sigma})=[x_0,\ldots,x_n]\in S_n^{\lamtil} (\widetilde{X}),
$$
where $x_i\in\lamtil$ is such that $\widetilde{\sigma} (e_i)\in \widetilde{B}_{x_i}$
for $i=0,\ldots, n$.

\begin{prop}\label{straight:prop}
The map $\strtil_\ast\colon C_\ast (\widetilde{X})\to C_\ast (\widetilde{X})$ satisfies the following properties:
\begin{enumerate}
\item
$d_{n+1}\circ \strtil_{n+1} = \strtil_{n}\circ d_{n+1}$ for every $n\in\matN$;
\item
$\strtil_n (\gamma\circ\widetilde{\sigma})=\gamma\circ \strtil_n(\widetilde{\sigma})$
for every $n\in\matN$, $\gamma\in\Gamma$, $\widetilde{\sigma}\in S_n (\widetilde{X})$;
\item
$\strtil_\ast (C_\ast (\widetilde{W}))
\subseteq C_\ast (\widetilde{W})$;
\item
the induced chain map $C_\ast (\widetilde{X},\widetilde{W})\to  
C_\ast (\widetilde{X},\widetilde{W})$, which we will still denote by $\strtil_\ast$, 
is $\Gamma$-equivariantly homotopic to the identity.
\end{enumerate}
\end{prop}
\begin{proof}
If $x_0,\ldots,x_n\in\widetilde{X}$, then 
it is easily seen that for every $i\leq n$ the $i$-th face of
$[x_0,\ldots,x_n]$ is given by
$[x_0,\ldots,\widehat{x}_i,\ldots,x_n]$; moreover since isometries preserve geodesics we
have $\gamma\circ [x_0,\ldots,x_n]=[\gamma (x_0),\ldots, \gamma (x_n)]$
for every $\gamma\in {\rm Isom} (\widetilde{X})$. Together with property~(2) in the definition of
net, these facts readily imply points~(1) and (2) of the proposition.

If ${\widetilde{\sigma}}\in S_n (\widetilde{W})$, then all the vertices of ${\widetilde{\sigma}}$
lie in the same connected component $\widetilde{K}$ of $\widetilde{W}$.
By property~(3) in the definition of net, the vertices of $\strtil_n ({\widetilde{\sigma}})$ still lie
in $\widetilde{K}$. Since $(X,W)$ is a locally convex pair, the 
subset $\widetilde{K}$ is convex in $\widetilde{X}$, so $\strtil_n ({\widetilde{\sigma}})$
belongs to $S_n (\widetilde{W})$, whence~(3). 

Finally, for ${\widetilde{\sigma}}\in S_n (\widetilde{X})$, 
let $F_{\widetilde{\sigma}}\colon \Delta_n\times [0,1]\to \widetilde{X}$
be defined by $F_{\widetilde{\sigma}}(x,t)=\beta_x (t)$, 
where $\beta_x\colon [0,1]\to \widetilde{X}$ is the constant-speed parameterization of
the geodesic segment joining ${\widetilde{\sigma}} (x)$ with 
$\strtil({\widetilde{\sigma}}) (x)$. We set $T_n({\widetilde{\sigma}})=(F_{\widetilde{\sigma}})_\ast (c)$,
where $c$ is the standard chain triangulating the prism $\Delta_n\times [0,1]$ by $(n+1)$-simplices.
The fact that $d_{n+1}T_n + T_{n-1}d_n={\rm Id}-\strtil_n$ is now easily checked,
while the $\Gamma$-equivariance of $T_\ast$ is a consequence of property~(2) of nets
together with the fact that geodesics are preserved by isometries. 
As above, the fact that $T_n(C_n(\widetilde{W}))\subseteq C_{n+1} (\widetilde{W})$
is a consequence of
the convexity of the components of $\widetilde{W}$.
\end{proof}

Let $\Lambda=p (\lamtil)$, and let
$S_\ast^\Lambda (X)$ 
be the subset of 
$S_\ast (X)$ given by those singular simplices which are obtained
by composing a simplex in $S_\ast^{\lamtil} (\widetilde{X})$ with the covering projection $p$. 
As a consequence of Proposition~\ref{straight:prop}
we get the following:

\begin{prop}\label{piecewise:prop}
The map  $\strtil_\ast$
induces a chain map $\str_\ast\colon C_\ast (X,W)\to 
C_\ast (X,W)$ which is homotopic to the identity.
\end{prop}

\begin{rem}
The maps $\strtil_\ast,\str_\ast$ obviously depend on the net chosen for their construction.
Such a dependence is however somewhat inessential in our arguments 
below.
Henceforth we understand that a net 
$\left(\lamtil, \{\widetilde{B}_x\}_{x\in\lamtil}\right)$
is fixed, and we denote
by $\strtil_\ast, \str_\ast$
the corresponding straightening operators.
\end{rem}

We are now ready to construct a chain map $\theta_*\colon \calC_\ast(X,W)\to C_\ast (X,W)$
whose induced map in homology will provide the desired norm non-increasing
inverse of $H_*(\iota_*)$.

Let us fix a simplex $\sigma\in S_n^\Lambda (X)$. It is readily seen
that the set $\str_n^{-1}(\sigma)$ is a Borel subset of $S_n (X)$. Therefore,
for every measure $\mu\in \calC_n (X)$ it makes sense to set
$$
c_\sigma(\mu)=\mu(\str_n^{-1}(\sigma))\ \in\R\ .
$$

\begin{lemma}\label{finite:lem}
 For every measure $\mu\in\calC_n (X)$, the set
$$
\{\sigma\in S_n^\Lambda (X)\, |\, c_\sigma(\mu)\neq 0\}
$$
is finite.
\end{lemma}
\begin{proof}
 Since $\mu$ admits a compact determination set, it is sufficient
to show that the family $\{\str_n^{-1}(\sigma),\, \sigma\in S_n^\Lambda (X)\}$
is locally finite in $S_n (X)$. So, let us take $\sigma_0\in S_n (X)$,
and let $\widetilde{\sigma}_0\in S_n(\widetilde{X})$ be a lift of
$\sigma_0$ to $\widetilde{X}$. For every $j=0,\ldots,n$, let
$Z_i$ be an open neighbourhood of $\widetilde{\sigma}_0(e_i)$ that intersects
only a finite number of $\widetilde{B}_{x_i}$'s, and let 
$\widetilde\Omega\subseteq S_n(\widetilde{X})$ be the set of $n$-simplices
whose $i$-th vertex belongs to $Z_i$ for every $i=0,\ldots,n$. Then
$\widetilde\Omega$ is an open neighbourhood of $\widetilde{\sigma}_0$ in
$S_n(\widetilde{X})$. 

Let $p_n\colon S_n(\widetilde{X})\to S_n (X)$ be the map taking
every $\widetilde{\sigma}\in S_n(\widetilde{X})$ into $p\circ\widetilde{\sigma}$.
It is proved in~\cite[Lemma A.4]{Frigerio} (see also~\cite{Loh})
that $p_n$ is a covering,
whence an open map, so $\Omega=p_n(\widetilde\Omega)$ is an open neighbourhood
of $\sigma_0$ in $S_n (X)$. 
Moreover, by construction the set $\str_n(\Omega)=\str_n(p_n(\widetilde\Omega))=
p_n(\strtil_n(\widetilde\Omega))$ is finite, whence the conclusion.
\end{proof}

By Lemma~\ref{finite:lem} we can define the map
$$
\theta_n\colon \calC_n(X)\to C_n(X),\qquad
\theta_n(\mu)=\sum_{\sigma\in S_n^\Lambda(X)} c_\sigma (\mu) \sigma \ .
$$

\begin{lemma}\label{thetalemma}
We have:
\begin{enumerate}
 \item 
$\theta_n\circ\partial_{n+1}=d_{n+1}\circ\theta_{n+1}$ for every $n\in\matN$.
\item
$\theta_n(\calC_n(W))\subseteq C_n(W)$ for every $n\in\matN$.
\item
$\|\theta_n(\mu)\|_1\leq \|\mu\|_\mea$ for every $\mu\in\calC_n(X)$, $n\in\matN$. 
\end{enumerate}
\end{lemma}
\begin{proof}
 Point~(1) is a direct consequence of the fact that $\str_\ast$ is a chain map.

Since $\str_n(C_n(W))\subseteq C_n(W)$, if $\sigma\in S_n^\Lambda (X)\setminus S_n(W)$,
then $\str_n^{-1}(\sigma)\cap S_n(W)=\emptyset$. Therefore, if $\mu\in\calC_n(W)\subseteq
\calC_n(X)$, then $c_\sigma (\mu)=\mu(\str_n^{-1}(\sigma))=0$, whence point~(2).

Point~(3) is a consequence of the fact that, if $\{Z_j\}_{j\in J}$
is a finite collection of pairwise disjoint Borel subsets of $S_n (X)$, then
$\sum_{j\in J} |\mu(Z_j)| \leq \|\mu\|_\mea$. 
\end{proof}

\subsection{Concluding the proof of Theorem~\ref{teo3}}\label{proofconvex:subsec}
As a consequence of Lemma~\ref{thetalemma}, the map
$\theta_\ast\colon \calC_\ast (X)\to C_\ast (X)$ induces 
norm non-increasing maps
$$
\overline{\theta}_\ast \colon \calC_\ast (X,W)\to C_\ast (X,W),\qquad
H_\ast(\overline{\theta}_\ast) \colon \calH_\ast (X,W)\to H_\ast (X,W)\ .
$$
Since we already know that $H_\ast(\iota_\ast)\colon H_\ast (X,W)\to \calH_\ast(X,W)$
is a norm non-increasing isomorphism, in order to prove that
$H_\ast(\iota_\ast)$ is an isometry it is sufficient to show
that $H_n(\overline{\theta}_\ast)\circ H_n(\iota_\ast)$ is the identity
of $H_n (X,W)$ for every $n\in\matN$. However, it follows by the very definitions
that $\overline{\theta}_n\circ \iota_n=\str_n$ for every $n\in \matN$, so the conclusion follows
from Proposition~\ref{piecewise:prop}.

\section{Relative bounded cohomology of groups}\label{boundedco:sec}
Let us recall some basic definitions and results about the bounded cohomology
of groups. For full details we refer the reader to~\cite{Gromov, Ivanov, Monod}.
Henceforth, we denote by $G$ a fixed group, which has to be thought as endowed with the 
discrete topology.
 
\begin{defn}[\cite{Ivanov, Monod}]
A \emph{Banach $G$-module} is a Banach space $V$ with a (left) action of $G$ such that
$\|g\cdot v\|\leq\|v\|$ for every $g \in G$ and every $v \in V$. A $G$-morphism
of Banach $G$-modules is a bounded $G$-equivariant linear operator.
\end{defn} 
From now on we refer to a Banach $G$-module  simply as a $G$-module.

\subsection{Relative injectivity}
A bounded linear map $\iota\colon A\rightarrow B$ of Banach spaces is 
\emph{strongly injective}
if there is a bounded linear map 
$\sigma\colon B\rightarrow A$ with $\|\sigma\|\leq 1$ and 
$\sigma\circ\iota={\rm Id}_A$ (in particular, $\iota$ is injective). 
We emphasize that, even when $A$ and $B$ are 
$G$-modules, the map $\sigma$ is \emph{not} required
to be  $G$-equivariant.

\begin{defn}\label{relativelyinjective}
A $G$-module $E$ is \emph{relatively injective} if for every 
strongly injective $G$-morphism $\iota\colon A\rightarrow B$
of Banach $G$-modules and every $G$-morphism $\alpha\colon A\rightarrow E$
there is a $G$-morphism $\beta\colon B\rightarrow E$ satisfying $\beta\circ \iota=
\alpha$ and $\|\beta\|\leq \|\alpha\|$.
$$
\xymatrix{
0 \ar[r] & A  \ar[r]_{\iota} \ar[d]_{\alpha} & B \ar@{->}[dl]^{\beta} \ar@/_/[l]_{\sigma}\\
& E
}
$$
\end{defn}

\subsection{Resolutions}
A $G$-complex (or simply a \emph{complex}) is a sequence of 
$G$-modules $E^i$ and $G$-maps $\delta^i\colon
E^i\to E^{i+1}$ such that $\delta^{i+1}\circ\delta^i=0$ for every $i$, where $i$ runs over
$\matN\cup\{-1\}$:

$$
0\longrightarrow E^{-1}\tto{\delta^{-1}} E^0\tto{\delta^0} E^1\tto{\delta^1}\ldots
\tto{\delta^n} E^{n+1}\tto{\delta^{n+1}}\ldots
$$
Such a sequence will often be denoted by $(E^\ast,\delta^\ast)$.

A $G$-\emph{chain map} (or simply a \emph{chain map}) between $G$-complexes $(E^\ast,\delta_E^\ast)$ and $(F^\ast,\delta_F^\ast)$
is a sequence of $G$-maps $\{\alpha^i\colon E^i\to F^i\, |\, i\geq -1\}$
such that $\delta_F^i\circ\alpha^i=\alpha^{i+1}\circ\delta_E^{i}$ for every $i\geq -1$. 
If $\alpha^\ast,\beta^\ast$ are chain maps between $(E^\ast,\delta_E^\ast)$ 
and $(F^\ast,\delta_F^\ast)$ which coincide in degree $-1$, a 
\emph{$G$-homotopy} 
between $\alpha^\ast$ and $\beta^\ast$ is 
a sequence of $G$-maps $\{T^i \colon E^i\to F^{i-1}\, |\, i\geq 0\}$ such that
$\delta_F^{i-1}\circ T^i + T^{i+1}\circ \delta_E^i=\alpha^i-\beta^i$   for every $i\geq 0$,
and $T^0\circ \delta_E^{-1}=0$. We recall that, according to our definition
of $G$-maps, both chain maps between $G$-complexes
and $G$-homotopies between such chain maps have to be bounded 
in every degree.

A complex is \emph{exact} if $\delta^{-1}$ is injective and 
$\ker \delta^{i+1}={\rm Im}\, \delta^i$ for every $i\geq -1$.
A $G$-\emph{resolution} (or simply a \emph{resolution}) of a $G$-module  $E$ 
is an exact $G$-complex $(E^\ast,\delta^\ast)$ with $E^{-1}=E$.
A resolution 
$(E^\ast,\delta^\ast)$ is \emph{relatively injective} 
if $E^n$ is relatively injective for every $n\geq 0$.

A \emph{contracting homotopy} for a resolution $(E^\ast,\delta^\ast)$ is a 
sequence of linear maps $k^i\colon E^i\to E^{i-1}$ such that 
$\|k^i\|\leq 1$ for every $i\in\matN$, 
$\delta^{i-1}\circ k^i+
k^{i+1}\circ\delta^i = {\rm Id}_{E^i}$ if $i\geq 0$, and $k^0 \circ \delta^{-1}=
{\rm Id}_E$. 
$$
\xymatrix{
0\ar[r] & E^{-1}\ar[r]_{\delta^{-1}} & E^0 \ar[r]_{\delta^0} \ar@/_/[l]_{k^0} &
E^1  \ar[r]_{\delta^1} \ar@/_/[l]_{k^1} & \ldots \ar@/_/[l]_{k^2} \ar[r]_{\delta^{n-1}}&
E^{n}
\ar[r]_{\delta^n} \ar@/_/[l]_{k^n} & \ldots \ar@/_/[l]_{k^{n+1}}
}
$$

Note however that it is not required that 
$k^i$ be $G$-equivariant. A resolution is \emph{strong}
if it admits a contracting homotopy.

The following result can be proved by means of standard homological algebra arguments
(see~\cite{Ivanov}, \cite[Lemmas 7.2.4 and 7.2.6]{Monod}).

\begin{prop}\label{ext:prop}
Let $\alpha\colon E\to F$ be a $G$-map between 
$G$-modules,
let $(E^\ast,\delta_E^\ast)$ be a strong resolution of $E$, and suppose
$(F^\ast,\delta_F^\ast)$ is a $G$-complex such that $F^{-1}=F$ and $F^i$ is relatively injective
for every $i\geq 0$. 
Then $\alpha$ extends to a chain map $\alpha^\ast$, and any two extensions
of $\alpha$ to chain maps are $G$-homotopic.
\end{prop}

\subsection{Absolute bounded cohomology of groups}
If $E$ is a $G$-module, we denote by $E^G\subseteq E$ the submodule
of $G$-invariant elements in $E$.

Let $(E^\ast,\delta^\ast)$ be a relatively injective strong resolution of the
trivial $G$-module $\R$ (such a resolution exists, see Subsection~\ref{standard:subsec}).
Since coboundary maps are $G$-maps, they restrict to the $G$-invariant submodules of the
$E^i$'s. Thus $((E^\ast)^G,\delta^\ast |)$ is a subcomplex of $(E^\ast,\delta^\ast)$. 
A standard application of Proposition~\ref{ext:prop} now shows that the isomorphism type of
the homology
of $((E^\ast)^G,\delta^\ast|)$ does not depend on the chosen resolution (while the seminorm induced on 
such homology module by the norms on the $E^i$'s could depend on it).
What is more, there exists a canonical isomorphism between the homology of any two such resolutions, which is induced
by any extension of the identity of $\R$. 
For every $n\geq 0$, we now define the $n$-dimensional \emph{bounded cohomology}
module $H_b^n (G)$ of $G$ as follows: if $n\geq 1$, then 
$H_b^n (G)$ is the $n$-th homology module of the complex $((E^\ast)^G,\delta^\ast |)$, 
while if $n=0$ then $H_b^n (G)=\ker \delta^0\cong \R$. 


\subsection{The standard resolution}\label{standard:subsec}
For every $n\in\matN$, let $B^n (G)$ be the space of bounded real maps  on $G^{n+1}$.
We endow $B^n(G)$ 
with
the supremum norm and with the diagonal action of $G$ defined by
$(g\cdot f)(g_0,\ldots,g_n)=f(g^{-1} g_0,\ldots, g^{-1}g_n)$, thus defining
on $B^n(G)$ a structure of $G$-module.
For $n\geq 0$ we define $\delta^n \colon B^n (G)\to B^{n+1} (G)$ by setting:
$$
\delta^n (f) (g_0,g_1,\ldots,g_{n+1})=\sum_{i=0}^{n+1} (-1)^i f(g_0,\ldots,\widehat{g}_i,\ldots,g_{n+1}).
$$
Moreover, we let
$B^{-1} (G)=\R$
be the trivial  $G$-module, and 
we define ${\delta}^{-1}\colon \R\to {B}^0 (G)$ by
setting ${\delta}^{-1} (t) (g)=t$ for every $g\in G$.

The following result is proved in~\cite{Ivanov}:

\begin{prop}\label{ivanov1}
The complex
$(B^\ast (G),{\delta}^\ast)$ provides a relatively injective strong
resolution of the trivial $G$-module $\R$.
\end{prop}

The resolution $(B^\ast (G),\delta^\ast)$ 
is usually known
as the \emph{standard resolution of the trivial $G$-module $\R$}. The seminorm 
induced on $H_b^\ast(G)$ by the standard resolution 
is called the \emph{canonical seminorm}. It is shown in~\cite{Ivanov}
that the canonical seminorm coincides with the infimum of all the seminorms
induced on $H_b^\ast(G)$ by any relatively injective strong resolution
of the trivial $G$-module $\R$ (see also Proposition~\ref{canonical seminorm} below).


\subsection{Relative bounded cohomology of groups}
Let $A$ be a subgroup of $G$. Henceforth, whenever $E$ is a $G$-module
we  understand that $E$ is endowed also with the natural structure of
$A$-module induced by the inclusion of $A$ in $G$.

\begin{defn}[Definitions 3.1 and 3.5 in \cite{Park}] \label{relativebc} 
Let $(U^\ast,\delta^\ast_U)$ be a relatively injective strong 
$G$-resolution of the trivial $G$-module $\matR$ and $(V^\ast,\delta^\ast_V)$ be a
relatively injective strong $A$-resolution of
the trivial $A$-module $\matR$. 
By Proposition~\ref{ext:prop}, the identity of $\R$ may be extended
to an $A$-chain map
$\lambda^\ast\colon U^\ast\rightarrow V^\ast$. The pair 
of resolutions $(U^\ast,\delta^\ast_U)$, $(V^\ast,\delta^\ast_V)$, together
with the chain map $\lambda^\ast$, provides a \emph{pair of resolutions for $(G,A;\R)$}.
We say that such a pair is
\begin{enumerate}
\item 
\emph{allowable}, if the chain map $\lambda^\ast$ commutes with 
the contracting homotopies of $(U^\ast,\delta^\ast_U)$ and
$(V^\ast,\delta^\ast_V)$;
\item
\emph{proper}, if the map
$\lambda^n$ restricts to a surjective map
$\widehat{\lambda}^n\colon(U^n)^G\rightarrow (V^n)^A$
for every $n\in\matN$.
\end{enumerate}
We denote by
$\ker (U^n \to V^n)$  the kernel of $\lambda^n$. 
It is readily seen that
the module $\ker (U^n\to V^n)^G\subseteq (U^n)^G$ coincides   
with the kernel of $\widehat{\lambda}^n$.
\end{defn}

If the pair of resolutions $(U^\ast,\delta^\ast_U)$, $(V^\ast,\delta^\ast_V)$ is proper,
there exists an exact sequence
$$
\xymatrix{
0\ar[r] & \ker (U^n\to V^n)^G\ar[r]& (U^n)^G\ar[r]^{\widehat{\lambda}^n} & (V^n)^A\ar[r] & 0 \ ,
}
$$
which induces the long exact sequence
\begin{equation*}
\xymatrix{
\ldots \ar[r] & H^{n-1}_b (A)\ar[r] & H^n (\ker (U^\ast \to V^\ast)^G)\ar[r] & H^n_b (G)\ar[r] & H_b^n(A)\ar[r]& \ldots
}
\end{equation*}

As observed in~\cite{Park}, if the pair
$(U^\ast,\delta^\ast_U)$, $(V^\ast,\delta^\ast_V)$ is also allowable, then
the isomorphism type of 
$H^n (\ker (U^\ast \to V^\ast)^G)$ does not depend
on the chosen proper allowable pair of resolutions 
(see also Proposition~\ref{canonical seminorm} below). Such a module is called the 
\emph{$n$-th bounded cohomology group of the pair $(G,A)$}, and it is denoted 
by $H_b^n(G,A)$. 

\subsection{The standard pair of resolutions}\label{canpair}
The following result is proved in~\cite[Propositions 3.1 and 3.18]{Park}, and shows 
that,
just as in the absolute case, there exists a canonical proper allowable pair
of resolutions for $(G,A;\R)$.

\begin{prop}
The standard resolutions $B^*(G)$ and $B^*(A)$
of the trivial $G$- and $A$-module $\mathbb{R}$, together with
the obvious restriction map $B^*(G)\to B^*(A)$,   
provide a proper allowable pair of resolutions for $(G,A;\matR)$. 
\end{prop}

The seminorm induced on $H_b^\ast(G,A;\R)$ by this resolution is called the
\emph{canonical seminorm}.
In order to save some words, from now on we fix the following notation:
$$
B^n(G,A)= \ker (B^n(G)\to B^n(A)).
$$

\subsection{Morphisms of pairs of resolutions}\label{morphisms}
Let $(U^*,\delta^\ast_U)$, $(V^*,\delta^\ast_V)$ and
$(E^*,\delta^\ast_E)$, $(F^*,\delta^\ast_F)$ be pairs of resolutions
for $(G,A;\R)$. A \emph{morphism} between such pairs is a pair 
of chain maps $(\alpha_G^\ast,\alpha_A^\ast)$ such that:
\begin{enumerate}
 \item 
$\alpha_G^\ast\colon U^\ast\to E^\ast$ (resp.~$\alpha_A^\ast\colon V^\ast\to F^\ast$)
is a $G$-chain map (resp.~an $A$-chain map)
extending the identity of $\R=U^{-1}=E^{-1}$
(resp.~the identity of $\R=V^{-1}=F^{-1}$);
\item
for every $n\in\matN$, the following diagram commutes
\begin{equation*}
\xymatrix{
U^n \ar[r] \ar[d]^{\alpha^n_G} &  V^n \ar[d]^{\alpha^n_A}\\
E^n \ar[r]  & F^n\ ,
}
\end{equation*}
where the horizontal rows represent the $A$-morphisms  involved in the definition
of a pair of resolutions.
\end{enumerate}

By condition~(2), if $(\alpha^\ast_G,\alpha^\ast_A)$ is a morphism of pairs of resolutions,
then $\alpha^\ast_G$ restricts to a chain map
$$
\alpha_{G,A}^\ast\colon \ker (U^\ast\to V^\ast)\to
\ker(E^\ast\to F^\ast)\ ,
$$
which induces in turn a map 
$$
H^\ast(\alpha^*_{G,A})\colon H^\ast(\ker(U^\ast\to V^\ast)^G)\to H^\ast(\ker(E^\ast\to F^\ast)^G)\ .
$$

\begin{prop}\label{buona:prop}
If the pairs of resolutions
$(U^*,\delta^\ast_U)$, $(V^*,\delta^\ast_V)$ and
$(E^*,\delta^\ast_E)$, $(F^*,\delta^\ast_F)$ are proper,
the map $H^\ast (\alpha^\ast_{G,A})$ is an isomorphism.
\end{prop}
\begin{proof}
Our hypothesis ensures that we have the commutative
diagram:
$$
\xymatrix{
\ldots 
H^{n-1} ((V^\ast)^A)\ar[r]\ar[d]^{H^{n-1}(\alpha^\ast_A)} 
& H^n (\ker (U^\ast \to V^\ast)^G)\ar[r]\ar[d]^{H^n(\alpha^\ast_{G,A})} & H^n ((U^\ast)^G)\ar[r]\ar[d]^{H^n(\alpha_G^\ast)} & H^n((V^\ast)^A)\ar[d]^{H^n(\alpha^\ast_A)}\ldots \\ 
\ldots 
H^{n-1} ((F^\ast)^A)\ar[r] & H^n (\ker (E^\ast \to F^\ast)^G)\ar[r] & H^n ((E^\ast)^G)\ar[r] & H^n((F^\ast)^A)\ldots 
}
$$
By Proposition~\ref{ext:prop}, the vertical arrows corresponding to
$H^\ast(\alpha^\ast_G)$ and $H^\ast(\alpha^\ast_A)$ are isomorphisms,
so the conclusion follows from the 
Five Lemma.
\end{proof}

\begin{rem}
At the moment we are not able to prove neither that every two proper allowable pairs
of resolutions for $(G,A;\R)$ are related by a morphism of pairs of resolutions,
nor that any two such morphisms induce the same map in cohomology. In fact, 
whenever two proper allowable pairs of resolutions are given, using Proposition~\ref{ext:prop}
one can easily construct the needed chain maps $\alpha^\ast_G$ and $\alpha^\ast_A$. However, some troubles arise
in proving that such chain maps can be chosen so to fulfill
condition~(2) in the above definition of morphism of pairs of resolutions.
Despite these difficulties, the results proved in Propositions~\ref{buona:prop}
and~\ref{canonical seminorm} are sufficient to our purposes.

Also observe that in the statement of Proposition~\ref{buona:prop}
we do not require the involved pairs of resolutions
to be allowable.
However, allowability plays a fundamental r\^ole in constructing
a morphism of pairs of resolutions between any generic proper allowable
pair of resolutions and the standard pair of resolutions (see Proposition~\ref{canonical seminorm} below), 
and in getting explicit bounds on the norm of such a morphism. 
\end{rem}

The following result shows that, just as in the absolute case, the bounded cohomology
of $(G,A)$ is computed by any proper allowable pair of resolutions for $(G,A;\R)$. Moreover,
the canonical seminorm coincides with the infimum of all the seminorms induced
on $H_b^\ast(G,A)$ by any such pair of resolutions.

\begin{prop}\label{canonical seminorm}
Let $(U^\ast,\delta_U^\ast)$, $(V^\ast,\delta_V^\ast)$ be a proper allowable
pair of resolutions for $(G,A;\R)$. Then there exists a morphism 
$(\alpha_G^\ast,\alpha_A^\ast)$ between this pair of resolutions and the canonical pair 
of resolutions introduced in Subsection~\ref{canpair}.
Moreover, 
one may choose $\alpha^*_G$, $\alpha^*_A$ in such a way that
the induced map $$H^\ast(\alpha_{G,A}^*)\colon H^\ast(\ker(U^\ast\to V^\ast)^G)\to H^\ast(B^\ast(G,A)^G)\cong H_b^*(G,A)$$
is a norm non-increasing isomorphism.
\end{prop}
\begin{proof}
Let $k_G^\ast$ (resp.~$k_A^\ast$) be the contracting homotopy of $(U^\ast,\delta_U^\ast)$
(resp.~of $(V^\ast,\delta_V^\ast)$).
Let us define $\alpha^n_G$ and $\alpha^n_A$ by induction as follows:
\begin{equation}\label{can:eq}
 \begin{array}{ccccc}
 \alpha_G^n(f)(g_0,\dots,g_n)&=&\alpha_G^{n-1}(g_0(k_G^n g_0^{-1}(f)))(g_1,\dots,g_n)
&\in& \matR\ ,\\
\alpha_A^n(f)(g_0,\dots,g_n)&=&\alpha_A^{n-1}(g_0(k_A^n g_0^{-1}(f)))(g_1,\dots,g_n)
&\in& \matR\ .
\end{array}
\end{equation}
It is easily seen that
$\alpha_G^\ast$ (resp.~$\alpha_A^\ast$) is indeed 
a $G$-chain map (resp.~an $A$-chain map) which is norm non-increasing in every
degree. Moreover, since the chain map $U^\ast\to V^\ast$ commutes with the contracting
homotopies of $(U^\ast,\delta^\ast_U)$ and $(V^\ast,\delta_V^\ast)$, the following
diagram commutes:
$$
\xymatrix{
U^n \ar[r] \ar[d]^{\alpha^n_G} & V^n \ar[d]^{\alpha^n_A}\\
B^n(G)\ar[r] & B^n(A)\ .
}
$$
This implies that $(\alpha^\ast_G,\alpha^\ast_A)$ is a morphism
of pairs of resolutions.
The conclusion follows from Proposition~\ref{buona:prop}.
\end{proof}

\section{Relative (continuous) bounded cohomology of spaces}\label{boundedco2:sec}
Throughout the whole section we denote by $(X,W)$
a countable CW-pair such that both $X$ and $W$ are connected. We also make
the assumption that the inclusion of $W$ in $X$ induces an injective
map on fundamental groups.

Being locally contractible, the space $X$ admits a universal covering
$p\colon \wdtX\rightarrow X$. We denote by $\wdtW$ a fixed connected
component of $p^{-1}(W)\subseteq \wdtX$. We also choose a basepoint
$b_0\in\wdtW$. This choice determines a canonical isomorphism
between $\pi_1(X,p(b_0))$ and the group $G$ of the covering
automorphisms of $\wdtX$. We denote by $A\subseteq G$ the subgroup
corresponding to $i_\ast(\pi_1(W,p(b_0)))$ under this isomorphism, where
$i\colon W\to X$ is the inclusion.
Observe that $A$ coincides with the group of automorphisms of $\wdtX$
that leave $\wdtW$ invariant.
In particular, for every $n\in\matN$ the module
$C^n_b(\wdtX)$ (resp.~$C^n_b(\wdtW)$)
admits a natural structure of $G$-module (resp.~$A$-module). 
Moreover, the covering projection $p\colon \wdtX\to X$ defines
a pull-back map $p^\ast\colon C_b^\ast(X,W)\to C_b^\ast(\wdtX,\wdtW)$ which induces
in turn an isometric isomorphism $C_b^\ast(X,W)\to C_b^\ast(\wdtX,\wdtW)^G$.
As a consequence,
we get the
natural identification
$$
H_b^\ast(X,W)\cong H^\ast(C_b^\ast(\wdtX,\wdtW)^G)\ .
$$

\smallskip

The straightening procedure described in Section~\ref{convex:sec}
shows that,
when $(X,W)$ is a locally convex pair of metric spaces, 
in order to compute 
the relative singular homology of $(X,W)$ one may replace
the singular complex $C_*(X,W)$ with the subcomplex of straight
chains. As a consequence, it is easily seen that 
in order to compute the cohomology (resp.~the bounded cohomology)
of $(X,W)$ one may replace the complex
$C^*(\wdtX,\wdtW)^G$ (resp.~$C_b^*(\wdtX,\wdtW)^G$) 
with the subcomplex of those invariant cochains
whose value on each simplex only depends on the vertices of the simplex
(recall that straight simplices in $\wdtX$ only depend on their vertices).
Following Gromov~\cite{Gromov}, we say that any such cochain is \emph{straight}.

Observe that the definition of straight cochain
makes sense even when it is not possible to properly define
a straightening on singular chains. Let us briefly describe
some known results about straight cochains in the absolute case
(\emph{i.e.}~when $W=\emptyset$). 
If $\wdtX$ is contractible, a classical result
ensures that both straight cochains and singular cochains
compute the cohomology of $G$, so
the cohomology of straight
cochains is isomorphic to the singular cohomology of $X$.
An important result 
by Gromov 
(see~\cite[Section 2.3]{Gromov} and~\cite[Theorem 4.4.1]{Ivanov}) 
shows that the same is true for bounded cohomology, even without the assumption
that $\wdtX$ is contractible. More precisely, both bounded straight cochains
and bounded singular cochains compute the bounded cohomology of $G$,
and they both induce the canonical seminorm on $H^*_b(G)$, 
so the cohomology of bounded straight cochains
is isometrically isomorphic to the bounded cohomology of $X$.
Moreover, Monod proved in~\cite[Theorem 7.4.5]{Monod}
that the bounded cohomology of $G$ (whence of $X$)
is computed also by
\emph{continuous} bounded straight cochains.
Monod's result plays a fundamental
r\^ole in L\"oh's description of 
the isometric isomorphism between measure homology and singular
homology in the absolute case.

In this section we show that, 
in the case when $W\neq\emptyset$, continuous bounded straight cochains 
compute the bounded cohomology of the pair $(G,A)$, thus extending
Monod's result to the relative case (see Theorem~\ref{isometric2:iso}).

Moreover, in the case when the pair $(X,W)$ is good we prove that
also $H_b^*(X,W)$ is isometrically isomorphic to $H_b^*(G,A)$, thus obtaining
that the bounded cohomology of $(X,W)$ is computed by continuous bounded straight cochains.
Finally, in Subsection~\ref{mainproof:sub} we show that this result easily implies
our Theorem~\ref{teo2}.


\subsection{Bounded cochains v.s.~continuous bounded straight cochains}\label{pairs:sub}
Let us give the precise definition of the complex of continuous bounded straight cochains.
For every $n\in\matN$ we consider the following Banach spaces:
$$
\begin{array}{lll}
C_{cbs}^n (\wdtX)&=&\{f\colon \wdtX^{n+1}\to \R,\ f\ {\rm continuous\ and\ bounded}\}\ ,\\
C_{cbs}^n (\wdtW)&=&\{f\colon \wdtW^{n+1}\to \R,\ f\ {\rm continuous\ and\ bounded}\}\ ,
\end{array}
$$
both endowed with the supremum norm. 
The diagonal $G$-action
such that $g\cdot f(x_0,\dots,x_n)=f(g^{-1}x_0 ,\dots,g^{-1}x_n)$ for every $g \in G$
endows $C^n_{cbs}(\wdtX)$ with a structure of $G$-module. 
The obvious coboundary maps
$$
\delta^n\colon C_{cbs}^n(\wdtX)\to C_{cbs}^{n+1}(\wdtX)\ ,\quad
\delta^n(f)(x_0,\ldots,x_{n+1})=\sum_{i=0}^{n+1}
(-1)^i f(x_0,\ldots,\widehat{x}_i,\ldots,x_{n+1})
$$
define on $C_{cbs}^\ast(\wdtX)$ a structure of $G$-complex. In the very same way
one endows $C_{cbs}^* (\wdtW)$ with a structure of $A$-complex.
For every $n\in\matN$, the inclusion $\wdtW^{n+1}\hookrightarrow \wdtX^{n+1}$ induces an obvious
restriction $C_{cbs}^n(\wdtX)\to C_{cbs}^n(\wdtW)$, whose kernel will be denoted by
$C_{cbs}^n(\wdtX,\wdtW)$. Finally, 
for every $n\in\matN$ we set
\begin{equation}\label{cbs}
H^n_{cbs} (X,W)=H^n(C^*_{cbs}(\wdtX,\wdtW)^G)\ .
\end{equation}

We will prove in Propositions~\ref{straightok} and~\ref{standardok}
that both $C_b^*(\wdtX)$, $C_b^*(\wdtW)$ and
$C_{cbs}^*(\wdtX)$, $C_{cbs}^*(\wdtW)$
provide proper pairs of resolutions for $(G,A;\R)$.
The pair of norm non-increasing chain maps
\begin{equation}\label{eta:eq}
\begin{array}{ll}
\eta_G^\ast\colon C_{cbs}^\ast (\widetilde{X})\to C_{b}^\ast (\wdtX),\ &
\eta_G^n (f)(\sigma)=f(\sigma(e_0),\ldots,\sigma (e_n))\ ,\\
\eta_A^\ast\colon C_{cbs}^\ast (\widetilde{W})\to C_{b}^\ast (\wdtW),\ & 
\eta_A^n (f)(\sigma)=f(\sigma(e_0),\ldots,\sigma (e_n))
\end{array}
\end{equation}
allows us to identify 
$C_{cbs}^\ast(\wdtX)$ (resp.~$C_{cbs}^\ast(\wdtW)$)
with the subcomplex of 
$C_b^*(\wdtX)$ (resp.~of $C_b^*(\wdtW)$)
of continuous bounded straight cochains on $\wdtX$ (resp.~on $\wdtW$).
Moreover, 
it is readily seen that the pair $(\eta_G^*,\eta_A^*)$ is a morphism of resolutions. Therefore, Proposition~\ref{canonical seminorm} implies that the induced map
in cohomology
$$
H^*(\eta_{G,A}^*)\colon H_{cbs}^*(X,W)=H^*(C^*_{cbs}(\wdtX,\wdtW)^G)\to H^*(C^*_b(\wdtX,\wdtW)^G)=H_b^*(X,W)
$$
is a norm non-increasing isomorphism.

Under the assumption that the pair $(X,W)$ is good, 
we are able to prove
that this isomorphism is in fact an isometry:

\begin{teo}\label{isometric:iso}
Suppose that $(X,W)$ is good. Then for every $n\in\matN$ the map
$$
H^n(\eta_{G,A}^*)\colon H_{cbs}^n(X,W)\to H_b^n(X,W)
$$
is an isometric isomorphism.
\end{teo}

Let us briefly sketch our strategy for proving 
Theorem~\ref{isometric:iso}. 
If $(X,W)$ is good, Proposition~\ref{standardok} below
ensures that bounded cochains provide a proper allowable pair of resolutions for $(G,A;\R)$,
so we may exploit
Proposition~\ref{canonical seminorm} to construct
a morphism of pairs of resolutions $(\alpha_G^\ast,\alpha_A^\ast)$
between bounded cochains
and the standard pair of resolutions for $(G,A;\R)$.
Then, in Subsection~\ref{constr}
we define a morphism of resolutions $(\beta_G^\ast,\beta_A^\ast)$ 
between the standard pair of resolutions and continuous bounded
straight cochains via an \emph{ad hoc} construction. 
Our assumptions imply that these morphisms induce  norm non-increasing isomorphisms
in cohomology, so in order to conclude
we will be left to show (in Subsection~\ref{isom:sub}) that
the composition $\beta_{G,A}^\ast\circ  \alpha_{G,A}^\ast$ induces 
the inverse of $H^\ast(\eta_{G,A}^\ast)$ in cohomology, \emph{i.e.}~that
the following diagram commutes:
$$
\xymatrix{
&H^\ast_b(G,A)\ar[ld]_{H^\ast(\beta^\ast_{G,A})} &  \\
H_{cbs}^\ast(X,W)\ar[rr]_{H^\ast(\eta_{G,A}^\ast)}
 & & H_b^\ast(X,W)\ar[lu]_{H^\ast(\alpha^\ast_{G,A})}\ .
}
$$

We begin with the following:

\begin{prop}\label{straightok} 
The pair $(C_{cbs}^\ast(\wdtX),\delta^\ast)$, $(C_{cbs}^\ast(\wdtW),\delta^\ast)$ provides
a proper allowable pair of resolutions for $(G,A;\R)$.
\end{prop}
\begin{proof}
The fact that  $(C_{cbs}^\ast(\wdtX),\delta^\ast)$
(resp.~$(C_{cbs}^\ast(\wdtW),\delta^\ast)$) provides a relatively injective
resolution of $\R$ as a trivial $G$-module (resp.~$A$-module)
is an immediate consequence of~\cite[Theorem 4.5.2]{Monod} 
(indeed in order to apply Monod's result our CW-complexes
$X$ and $W$ should be locally compact, whence locally finite, 
but these conditions are used in the proof of~\cite[Theorem 4.5.2]{Monod}
only to ensure the existence of a suitable \emph{Bruhat function} on $\wdtX$ and 
on $\wdtW$; in our case
of interest the fact that $G$ and $A$ are discrete allows us to explicitly describe
such a map -- see  Lemma ~\ref{mappah}).

Moreover, it is readily seen that
these resolutions
admit the following contracting homotopies:
\begin{equation}\label{contract}
\begin{array}{ll}
t^{n}_G(f)(x_1,\dots,x_n)=f(b_0,x_1,\dots, x_n)\ , & f\in C_{cbs}^{n}(\wdtX),\ 
(x_1,\ldots,x_n)\in \wdtX^{n},\\
t^{n}_A(f)(w_1,\dots,w_n)=f(b_0,w_1,\dots, w_n)\ , & f\in C_{cbs}^{n}(\wdtW),\ 
(w_1,\ldots,w_n)\in \wdtW^{n}\ .
\end{array}
\end{equation}
This readily implies that the $A$-chain map 
$\gamma^\ast\colon C^\ast_{cbs}(\wdtX)\to C^\ast_{cbs}(\wdtW)$ induced by the inclusion
$\wdtW\hookrightarrow\wdtX$ commutes with the contracting homotopies.

In order to conclude we have to show that $\gamma^\ast$ restricts
to a surjective map
$$
\hatgamma^\ast\colon  C_{cbs}^\ast(\widetilde{X})^G\longrightarrow 
  C_{cbs}^\ast(\widetilde{W})^A\ .
$$
Let  $f\colon\wdtW^{n+1}\rightarrow \matR$ be an 
$A$-invariant bounded continuous map.
The inclusion $\wdtW^{n+1}\hookrightarrow \wdtX^{n+1}$
induces a homeomorphism $\psi$ between $\wdtW^{n+1}/A$ and
a closed subset $K$ of $\wdtX^{n+1}/G$ (recall that $W$ is a CW-subcomplex
of $X$, so it is closed in $X$). Therefore, $f$ defines
a bounded continuous map $\overline{f}$ on $K$, and by Tietze's Theorem
we may extend $\overline{f}$ to a bounded continuous map
$\overline{g}\colon \wdtX^{n+1}/G\to\R$. If $g$ is obtained by precomposing
$\overline{g}$  with the projection $\wdtX^{n+1}\to \wdtX^{n+1}/G$,
then ${g}\in C_{cbs}^n(\wdtX)^G$, and 
$\hatgamma^n(g)=f$. We have thus shown that
$\hatgamma^\ast$ is surjective, and this concludes the proof.
\end{proof}

\subsection{Ivanov's contracting homotopy}\label{ivanov:sub}
In order to show that, under the hypothesis that
$(X,W)$ is good, bounded cochains provide a proper allowable pair
of resolutions for $(G,A;\R)$, we first
recall Ivanov's construction of a contracting homotopy for 
the resolution $C_b^\ast(\wdtX)$. 

It is shown in~\cite{Ivanov} that one can construct an infinite Postnikov system
$$
\xymatrix{
\ldots \ar[r]^{p_m} & X_m \ar[r]^(.4){p_{m-1}} & X_{m-1}\ar[r]^{p_{m-2}} & \ldots & \ar[r]^{p_2} & X_2 \ar[r]^{p_{1}} & X_1\ ,
}
$$
where $X_1=\wdtX$, $\pi_i(X_m)=0$ for every $i\leq m$, $\pi_i(X_m)=\pi_i(X)$ for every $i>m$ and each map
$p_m\colon X_{m+1}\to X_m$ is a principal $H_m$-bundle for some topological connected abelian group
$H_m$, which has the homotopy type of a $K(\pi_{m+1}(X),m)$.
Moreover, the induced chain maps $p^\ast_m\colon C_b^\ast (X_m)\to C_b^*(X_{m+1})$ admit left inverse chain maps
$A_m^\ast\colon C_b^\ast (X_{m+1})\to C_b^*(X_{m})$ obtained by averaging cochains
over the preimages in $X_{m+1}$
of simplices in $X_m$, in such a way that the $A_m$'s are norm non-increasing. 

Let us now denote by $W_m\subseteq X_m$ the preimage $p_{m-1}^{-1}(p_{m-2}^{-1}(\ldots(p_1^{-1}(\wdtW))))\subseteq X_m$
(so $W_{m+1}$ is a principal $H_m$-bundle over $W_m$ for every $m\geq 1$). We denote simply by $p_m
\colon W_{m+1}\to W_m$ the restriction of $p_m$ to $W_{m+1}$. It follows
from Ivanov's construction that each $A_m^\ast$
induces a norm non-increasing chain map $C_b^\ast (W_{m+1})\to C_b^*(W_{m})$, which will still be denoted by
$A_m^\ast$.

\begin{lemma}\label{Wn}
Suppose that $(X,W)$ is good. Then
$\pi_i (W_m)=0$ for every $i\leq m$.
\end{lemma}
\begin{proof}
Of course, it is sufficient to prove that $\pi_i (W_m)\cong \pi_i (X_m)$ for every $i\in\matN$, $m\in\matN$.
Let us prove this last statement by induction on $m$.
Since the inclusion map $W\hookrightarrow X$ is $\pi_1$-injective
we have $\pi_1(W_1)=\pi_1(X_1)=0$. Therefore, since coverings induce isomorphisms on homotopy groups of order
at least two,
the case $m=1$ follows from the fact that
the pair $(X,W)$ is good.
The inductive step 
follows from an easy application of the Five Lemma to the
following commutative diagram, which descends in turn  from the naturality
of the homotopy exact sequences for  the bundles  $X_{m+1}\to X_m$, $W_{m+1}\to W_m$:
$$
\xymatrix{
\pi_{i+1}(W_m)\ar[r]\ar[d] & \pi_i(H_m)\ar[r] \ar@{=}[d]& \pi_i (W_{m+1})\ar[r] \ar[d] & \pi_i(W_m)\ar[r] \ar[d] & \pi_{i-1} (H_m)\ar@{=}[d] \\
\pi_{i+1}(X_m)\ar[r] &
 \pi_i(H_m)\ar[r] & \pi_i (X_{m+1})\ar[r] & \pi_i(X_m)\ar[r] & \pi_{i-1} (H_m)\ .
}
$$
\end{proof}


Let us now suppose that $(X,W)$ is good.
We choose basepoints $w_m\in W_m$ in such a way that $p_m(w_{m+1})=w_m$ for every $m\geq 1$, and $w_1\in W_1=\wdtW$ coincides with the basepoint $b_0$
fixed above. 
Since $X_m$ is $m$-connected,
for every $n\leq m$ it is possible to construct
a map $L^m_n\colon S_n(X_m)\to S_{n+1}(X_m)$ that associates to every $\sigma\in S_n(X_m)$
a \emph{cone} of $\sigma$ over $w_m$ (see \cite{Ivanov}). We stress the fact that, since $W_m$ is also $m$-connected,
if $\sigma\in S_n(W_m)\subseteq S_n(X_m)$, then $L^m_n(\sigma)$ can be chosen
to belong to $S_{n+1}(W_m)$. 
The maps $L_n^m$, $n\leq m$, induce a (partial) homotopy between 
the the identity and the null map of $C_\ast (X_m)$, which induces in turn a (partial)
contracting homotopy $\{k_m^n\}_{n\leq m}$ for the (partial) complex $\{C_b^n(X_m)\}_{n\leq m}$.
Since $L^m_n(S_n(W_m))\subseteq S_{n+1}(W_m)$, this contracting homotopy induces a 
(partial) contracting homotopy for $\{C_b^n(W_m)\}_{n\leq m}$, which we still
denote by $k_m^*$.
Moreover, it is possible to choose these contracting homotopies in a compatible way, in the sense
that the equality
$A_m^{n-1}\circ k^{n}_{m+1} \circ p_m^n =k_m^n$ holds for every $n\leq m$
(see again~\cite{Ivanov}). Thanks to this compatibility condition,
one can finally define the contracting homotopy
$$
k^\ast_G\colon C^\ast_b(\wdtX)\to C_b^{\ast-1} (\wdtX),\qquad
$$
via the formula
$$
k^n_G=A^{n-1}_1 \circ \ldots \circ A^{n-1}_{m-1} \circ k_{m}^n \circ p_{m-1}^n   \circ \ldots \circ p_2^n \circ  p_1^n,\qquad {\rm for\ any}\ m\geq n\ .
$$
The very same formula defines a contracting homotopy for $C^*_b(\wdtW)$. By construction, the restriction map
$C_b^\ast(\wdtX)\to C_b^\ast(\wdtW)$ commutes with these contracting homotopies, and it obviously
restricts to a surjective map $C_b^\ast(\wdtX)^G\to C_b^\ast(\wdtW)^A$.
Since $C^n_b(\wdtX)$,
$C^n_b(\wdtW)$ are relatively injective for every $n\geq 0$ (see~\cite{Ivanov}), we have finally proved the following:

\begin{prop}\label{standardok}
The pair $(C_b^\ast(\wdtX),\delta^\ast)$, $(C_b^\ast(\wdtW),\delta^\ast)$ provides
a proper pair of resolutions for $(G,A;\R)$. If in addition
$(X,W)$ is good, then this pair of resolutions is also allowable.
\end{prop}

\begin{rem}\label{Park:rem}
The fact that the pair of resolutions
$(C_b^\ast(\wdtX),\delta^\ast)$, $(C_b^\ast(\wdtW),\delta^\ast)$ is allowable
 is stated  in~\cite[Lemma 4.2]{Park} under the only assumption that $(X,W)$
is a pair of connected CW-pairs. However, at the moment we are not able to prove
such a statement without the assumption that $(X,W)$ is good. For example,
let us suppose that $X$ is simply connected and $W$ is a point
(so that $\pi_n(W)$ injects into $\pi_n(X)$ for every $n\in\matN$,
and $X_1=\wdtX=X$, $W_1=\wdtW=W$). 
Then for every $n\in\matN$ there exists only one simplex
in $S_n(W)$, namely the constant $n$-simplex $\sigma_n^W$. 
Therefore, the only possible contracting homotopy for $W$ is given by the map which
sends the cochain $\varphi\in C^n_b(W)$ to the cochain $k^n_A(\varphi)$ such
that $k^n_A(\varphi)(\sigma_{n-1}^W)=\varphi(\sigma_n^W)$. 
On the other hand, 
it is not difficult to show that  
$\pi_i(W_m)=\pi_{i+1}(X)$ for every $i< m$, and
$\pi_i(W_m)=0$ for every $i\geq m$. Therefore, 
if $\pi_{i+1}(X)\neq 0$, then $\pi_i (W_m)\neq 0$ for every $m>i$. This readily implies
that for $m>i$ 
one cannot construct \emph{cone-like} 
operators
$L^m_j\colon C_j(X_m)\to C_{j+1}(X_m)$, $j\leq i$, 
such that $d_{j+1}L^m_j+L^m_{j-1}d_j={\rm Id}$ and
$L^m_j(C_j(W_m))\subseteq C_{j+1}(W_m)$ for every $j\leq i$, so it is not clear how to 
show that the pair of resolutions $C^\ast_b(\wdtX)$, $C^\ast_b(\wdtW)$ is allowable.
This difficulty already arises for the pair $(S^2,q)$, where $q$ is any point of the $2$-dimensional sphere $S^2$.

Some troubles arise also in the case when the inclusion
induces surjective (but not bijective) maps between the homotopy
groups of $W$ and of $X$. For instance, if
$X$ is the Euclidean $3$-space and $W=S^2$, then $X_m=X$ for every $m\in\matN$,
so $W_m=W$ for every $m\in\matN$, and, if $i$ is sufficently
high, the partial complex $\{C_j(X,W)\}_{j\leq i}$ does not support
a relative cone-like operator. Also observe that, if $\{W_m',\,\ m\in\matN\}$
is a Postnikov system over $W$, then the  
only map $W_m'\to W_m=S^2\subseteq \R^3=X_m$
which commutes with the projections of $W_m'$ and $X_m$ onto
$W_1=S^2$ and $X_1=\R^3$ is the projection
$W_m'\to W_1=S^2$. As a consequence, also in this case it is not clear
why the pair of resolutions $C^\ast_b(\wdtX)$, $C^\ast_b(\wdtW)$ should be allowable.
\end{rem}


\subsection{Mapping bounded cochains into the standard resolutions}
Throughout the whole subsection we suppose that $(X,W)$ is good. 
By Proposition~\ref{standardok}, under this assumption 
Proposition~\ref{canonical seminorm} provides a morphism of pairs of resolutions
 \begin{equation}\label{alpha:eq}
 \alpha_G^\ast\colon C^\ast_b(\wdtX)\to B^\ast(G),\qquad
 \alpha_A^\ast\colon C^\ast_b(\wdtW)\to B^\ast(A)\ 
 \end{equation}
such that the induced map $H^\ast(\alpha_{G,A}^\ast)$ is a
norm non-increasing isomorphism.
The definition of the chain maps $\alpha_G^\ast$, $\alpha_A^\ast$ 
involve the contracting homotopies for the resolutions
 $C^\ast_b(\wdtX)$, $C^\ast_b(\wdtW)$ described in Subsection~\ref{ivanov:sub}.
Being based on a non-explicit averaging procedure,
such contracting homotopies cannot be described by an explicit formula, and the same
is true for the chain maps $\alpha_G^\ast$, $\alpha_A^\ast$. However,
in order to show that
the composition $\beta_{G,A}^\ast\circ \alpha_{G,A}^\ast$ induces 
the inverse of $H^\ast(\eta_{G,A}^\ast)$ in cohomology, the following
explicit description of the composition
$\alpha_{G,A}^\ast\circ \eta_{G,A}^\ast$ will prove sufficient:
\begin{lemma}\label{esplicita3}
Suppose that $(X,W)$ is good.
For every $f\in C_{cbs}^n(\wdtX,\wdtW)$ we have
$$
\alpha_{G,A}^n (\eta_{G,A}^n(f)) (g_0,\ldots,g_n)= f(g_0 b_0,\ldots, g_n b_0)
\ .
$$
\end{lemma}
\begin{proof}
Let $t^\ast_G$ (resp.~$k^\ast_G$) be the contracting homotopy for continuous 
bounded straight cochains (resp.~for bounded cochains) described in Equation~\eqref{contract}
(resp.~in Subsection~\ref{ivanov:sub}).
We begin by showing that
for every $n\in\matN$ we have
\begin{equation}\label{commuta}
k_G^n\circ \eta^n_G=\eta^{n-1}_G\circ t_G^n\ .
\end{equation}
Let us fix
$f\in C^{n}_{cbs}(\wdtX)$ and $\sigma\in S_{n-1}(\wdtX)$, and let us compute
$k_G^n(\eta_G^n(f))(\sigma)$.
With notations as in Subsection~\ref{ivanov:sub}, 
we choose $m\geq n$ and set $$f_m=p_{m-1}^n(\ldots p_1^n(\eta_G^n(f)))\in C_b^n(X_m)\ .$$ Then, 
if $\sigma_m$ is any lift of $\sigma$ in $X_m$, we have
$k_m^{n}(f_m) (\sigma_m)=f_m(\sigma_m')$, where
 $\sigma_m'\in S_{n}(X_m)$ has vertices $w_m, \sigma_m(e_0),\ldots,\sigma_m(e_{n-1})$.
 It readily follows that $$k_m^n(f_m)(\sigma_m)=f(b_0,\sigma(e_0),\ldots,\sigma (e_{n-1}))\ .$$
 We have thus shown that the cochain $k_m^{n}(f_m)$ is constant
 on all the lifts of $\sigma$ in $X_m$. 
 By definition, the value of $k_G^{n}(\eta_G^n(f))(\sigma)$
is obtained by suitably averaging the values taken by  $k_m^{n}(f_m)$ 
on such lifts, so we finally get
$$
k_G^{n}(\eta_G^n(f)) (\sigma)=f (b_0, \sigma (e_0),\ldots, \sigma (e_{n-1}))\ ,
$$
whence Equation~\eqref{commuta}.

Recall now that the composition $\alpha_{G,A}^\ast\circ \eta_{G,A}^\ast$
is obtained by restricting the map $\alpha_G^*\circ \eta_G^*$, where
$\alpha_G^*$ is explicitely described (in terms of the contracting homotopy
$k_G^\ast$) in Proposition~\ref{canonical seminorm} (see~Equation~\eqref{can:eq}).
Therefore, Equations~\eqref{can:eq} and~\eqref{commuta} 
readily imply that  the composition $\alpha_G^n\circ \eta_G^n$
can be described by the following inductive formula:
$$
\alpha_G^n (\eta_G^n(f)) (g_0,\ldots,g_n)=
\alpha_G^{n-1}(g_0(\eta^{n-1}_G(t_G^{n}(g_0^{-1}(f)))))(g_1,\dots,g_n)\ .
$$
An easy induction implies the conclusion.
\end{proof}

\subsection{Mapping the standard resolutions into continuous bounded straight cochains}\label{constr}
In this subsection we do not assume that the pair $(X,W)$ is good.
In order to define a morphism of pairs of resolutions between
the standard 
pair of resolutions
for $(G,A;\R)$ and 
the complex of continuous bounded straight cochains we need the following result,
which generalizes~\cite[Lemma 5.1]{Frigerio}:
\begin{lemma}\label{mappah}
There exists a continuous map
$\chi\colon \wdtX\rightarrow [0,1]$  with the following properties:
\begin{enumerate}
\item For every $x \in \wdtX$ there exists a neighbourhood $U_x$ of 
$x \in \wdtX$ such that the set
$\{g\in G\;|\;{\rm supp}(\chi)\cap g(U_x)\neq \emptyset\}$
is finite.
\item For every $x\in \wdtX$, we have
$\sum_{g\in G} \chi (g\cdot x)=1$
(Note that the sum on the left-hand side is finite by (1)).
\item
For every $w\in\widetilde{W}$ and every $g\in G\setminus A$, we have
$\chi (g\cdot w)=0$, whence
$\sum_{g\in A} \chi (g\cdot w)=1$.
\item
We have $\chi (b_0)=1$, so 
$\chi (g\cdot b_0)=0$ for every $g\neq 1$.
\end{enumerate}
\end{lemma}
\begin{proof}
Let $\mathcal{U}=\{U_i\}_{i\in I}$ be
a locally finite covering of $X$ with evenly-covered neighborhoods
(with respect to the universal covering $\wdtX\rightarrow X$). Since
$W$ is a subcomplex of $X$, we may also suppose that
the intersection of $W$ with each $U_i$ is connected. 
We choose $i_0\in I$ such that $p(b_0)$ belongs to $U_{i_0}$,
and we replace each $U_i$, $i\neq i_0$, with $U_i\setminus \{p(b_0)\}$.
Let now
$J=\{i\in I\, |\, U_i\cap W\neq \emptyset\}$
(so $i_0\in J$).

For every $U_i$ let us choose an open subset $H_i\subseteq \wdtX$
in such a way that the following conditions hold:
\begin{itemize}
 \item 
 $p|_{H_i}\colon H_i\to U_i$ is a homeomorphism;
\item
$p^{-1}(U_i)=\bigcup_{g\in G}g(H_i)$ 
and $g(H_i)\cap g'(H_i)=\emptyset$
for every $g\neq g'$;
\item
$H_i\cap \wdtW\neq\emptyset$ for every $i\in J$.
\end{itemize}
Since $U_i\cap W$ is connected, the last condition easily implies
that 
\begin{equation}\label{H:eq}
H_i\cap p^{-1}(W)=H_i\cap \wdtW\qquad  {\rm for\ every}\ i\in I\ . 
\end{equation}
Since every CW-complex is paracompact (see \emph{e.g.}~\cite{para, para2}),
we may take a partition of unity $\{\varphi_i\}_{i\in I}$  
adapted to $\mathcal{U}$, and let
$\psi_i\colon \wdtX\rightarrow \matR$ be the map 
that coincides with $\varphi_i\circ p $ on $H_i$ and is null outside $H_i$.
We finally set
$$
\chi=\sum_{i \in I}\psi_i \ . 
$$
The fact that $\chi$ satisfies properties~(1) and~(2)
of the statement is proved in~\cite[Lemma 5.1]{Frigerio}.
Moreover, Equation~\eqref{H:eq} implies that for every $w\in\wdtW$ and $g\in G\setminus A$
we have $g\cdot w\notin \wdtW$, so 
$g\cdot w$ does not belong to any $H_i$, whence point~(3).
Finally, since $p(b_0)\notin U_i$ for every $i\neq i_0$, we have necessarily
$\psi_i(b_0)=0$ for every $i\neq i_0$, whence
$\psi_{i_0}(b_0)=1$ and $\chi (b_0)=1$.
\end{proof}

We are now ready to describe a morphism of pairs of resolutions
$(\beta_G^\ast,\beta_A^\ast)$ between 
the standard pair of resolutions for $(G,A;\R)$ and
the complexes of straight cochains. Let  
\begin{equation}\label{beta:eq}
\beta_G^n\colon B^n(G)\longrightarrow C^n_{cbs}(\wdtX),\qquad
\beta_A^n\colon B^n(A)\longrightarrow C^n_{cbs}(\wdtW)\ 
\end{equation}
be defined as follows:
$$
\begin{array}{lll}
\beta_G^n(f)(x_0,\dots,x_n)&=&\sum_{(g_0,\dots,g_n)\in 
G^{n+1}}\chi(g_0^{-1}x_0)
\cdots \chi(g_n^{-1}x_n)\cdot f(g_0,\dots,g_n)\ ,\\
\beta_A^n(f)(w_0,\dots,w_n)&=&\sum_{(g_0,\dots,g_n)\in 
A^{n+1}}\chi(g_0^{-1}w_0)
\cdots \chi(g_n^{-1}w_n)\cdot f(g_0,\dots,g_n)\ .
\end{array}
$$
It is proved in~\cite[Proposition 5.5]{Frigerio} that $\beta^\ast_G$ is a well-defined
chain map that extends the identity of $\R$. Using point~(3) of Lemma~\ref{mappah}, 
it is easy to show that the same is true
for $\beta^\ast_A$, and that 
$(\beta_G^\ast,\beta_A^\ast)$ indeed provides 
a morphism of pairs of resolutions.
Moreover, $\beta_G^n$ is norm non-increasing for every $n\in\matN$. Therefore, 
Proposition~\ref{buona:prop}
readily implies that $H^\ast(\beta^\ast_{G,A})$ is a norm non-increasing isomorphism. 




\subsection{Proof of Theorem \ref{isometric:iso}}\label{isom:sub}
Let us suppose that $(X,W)$ is good, and let us come back to
the diagram
$$
\xymatrix{
&H^\ast_b(G,A)\ar[ld]_{H^\ast(\beta^\ast_{G,A})} &  \\
H_{cbs}^\ast(X,W)\ar[rr]_{H^\ast(\eta_{G,A}^\ast)}
 & & H_b^\ast(X,W)\ar[lu]_{H^\ast(\alpha^\ast_{G,A})}\ .
}
$$
We already know that $H^\ast(\alpha_{G,A}^\ast)$,
$H^\ast(\beta_{G,A}^\ast)$ 
and $H^\ast(\eta_{G,A}^*)$ are norm non-increasing isomorphisms. 
Therefore, in order to conclude the proof of Theorem~\ref{isometric:iso}
we are left to show that
the above diagram commutes, 
\emph{i.e.}~that  $H^\ast(\alpha_{G,A}^\ast) \circ H^\ast (\eta_{G,A}^\ast)\circ H^\ast(\beta_{G,A}^\ast) $ is equal to the identity of
$H_b^*(G,A)$.

Let us take $f\in B^n(G,A)$. By Lemma~\ref{mappah}-(4), for every $(\gamma_0,\ldots,\gamma_n)\in G^{n+1}$,
$(g_0,\dots, g_n)\in G^{n+1}$ we have
$$
\chi(\gamma_0^{-1}g_0b_0)\cdots \chi (\gamma_n^{-1}g_nb_0) \cdot
f(\gamma_0,\ldots,\gamma_n)=\left\{
\begin{array}{ll} f(g_0,\ldots,g_n)\  & {\rm if}\ \gamma_i=g_i\ {\rm for\ every}\ i\\
0 \ & {\rm otherwise} \end{array}\right. \ ,
$$
and this readily implies that 
$$
\beta_{G,A}^n (f)(g_0b_0, \ldots, g_nb_0)=f(g_0,\ldots,g_n)\ .
$$
Putting together this equality with 
Lemma~\ref{esplicita3} we readily get 
$$
\alpha_{G,A}^n(\eta_{G,A}^n(\beta_{G,A}^n(f)))(g_0,\ldots,g_n)  = 
\beta^n_{G,A} (f) (g_0b_0,\ldots, g_nb_0)
=f(g_0,\ldots,g_n), \
$$
so $\alpha_{G,A}^n\circ \eta_{G,A}^n\circ \beta_{G,A}^n$ is the identity already at the level of cochains,
whence the conclusion.

\medskip

It is maybe worth stressing the fact that continuous bounded straight cochains
compute the bounded cohomology of the pair $(G,A)$ even
without the assumption that 
$(X,W)$ is a good pair.
More precisely we have:

\begin{teo}\label{isometric2:iso}
For every $n\in\matN$ the map
$$
H^*(\beta^*_{G,A})\colon 
H_b^n(G,A)\to H^n_{cbs}(X,W)$$ 
is an isometric isomorphism.
\end{teo}
\begin{proof}
Recall that continuous bounded straight cochains provide a proper
allowable pair of resolutions for $(G,A;\R)$ even when
the pair $(X,W)$ is not good.
Therefore, the construction carried out in
Subsection~\ref{constr}
provides a norm non-increasing isomorphism
$$H^\ast(\beta^\ast_{G,A})\colon H^\ast_b(G,A)\rightarrow H^\ast_{cbs}(X,W)\ ,$$
and Proposition~\ref{canonical seminorm} provides a morphism of pairs of resolutions
$$
\widehat{\alpha}_G^\ast\colon C^\ast_{cbs}(\wdtX)\to B^\ast(G),\qquad
\widehat{\alpha}_A^\ast\colon C^\ast_{cbs}(\wdtW)\to B^\ast(A)\ 
$$
that induces 
a norm non-increasing isomorphism $H^\ast(\widehat{\alpha}_{G,A}^\ast)\colon H^\ast_{cbs}(X,W)\to H^\ast_b(G,A)$. Just as in the 
proof of Theorem \ref{isometric:iso}, in order to conclude
it is sufficient to show that for every $n\in\matN$
the composition $\widehat{\alpha}_{G}^n\circ \beta_{G}^n$
is the identity of $B^n(G)$.

Recall from Proposition~\ref{canonical seminorm} 
that  the map $\widehat{\alpha}_G^n$
can be described by the following inductive formula:
$$
\widehat{\alpha}_G^n (f) (g_0,\ldots,g_n)=
\widehat{\alpha}_G^{n-1}(g_0(t_G^{n}(g_0^{-1}(f))))(g_1,\dots,g_n)\ ,
$$
where $t^\ast_G$ is the contracting homotopy for the resolution $C^\ast_{cbs}(\wdtX)$
described in Equation~\eqref{contract}.
As a consequence, an easy induction shows that $\widehat{\alpha}_G^n(f)(g_0,\ldots,g_n)=
f(g_0b_0,\ldots, g_nb_0)$ for every $f\in C_{cbs}^n(\wdtX)$, $(g_0,\ldots,g_n)\in G^{n+1}$,
and this implies in turn that 
$\widehat{\alpha}_{G}^n\circ \beta_{G}^n$
is the identity of $B^n(G)$, whence the conclusion.
\end{proof}

\subsection{Proof of Theorem \ref{teo2}}\label{mainproof:sub}
In this subsection we describe how 
Theorem~\ref{teo2} can be deduced from
Theorem~\ref{isometric:iso}. 
For every $n\in \mathbb{N}$ the module $C_{cb}^n(\wdtX)$
(resp.~$C_{cb}^n(\wdtW)$) admits a natural structure of $G$-module (resp.~$A$-module). 
Moreover, it is proved in~\cite[Lemma 6.1]{Frigerio} that
the isometric isomorphism $C_b^\ast(X,W)\to C_b^\ast(\wdtX,\wdtW)^G$
induced by the 
covering projection $p\colon \wdtX\rightarrow X$ 
restricts to
an isometric isomorphism $C_{cb}^*(X,W)\to C_{cb}^*(\wdtX,\wdtW)^G$, which induces
in turn 
a natural identification
\begin{equation}\label{identifications}
H_{cb}^\ast(X,W)\cong H^\ast(C_{cb}^\ast(\wdtX,\wdtW)^G)\ .
\end{equation}

The $G$-chain map $\nu_G^\ast:C^\ast_{cbs}(\wdtX)\rightarrow C^\ast_{cb}(\wdtX)$
defined by
$$
 \nu_G^n (f)(\sigma)=f(\sigma(e_0),\ldots,\sigma (e_n))\qquad  {\rm for\ every}\ n\in\matN,\ f\in C^n_{cbs}(\wdtX),\ \sigma\in S_n (\wdtX),
$$
obviously restricts to a chain map
$\nu^*_{G,A}\colon C^\ast_{cbs}(\wdtX,\wdtW)^G\to C^\ast_{cb}(\wdtX,\wdtW)^G$.
Under the identifications  described in Equations~\eqref{cbs} and \eqref{identifications},
this chain map induces the norm non-increasing map
$$
H^\ast(\nu_{G,A}^\ast)\colon H_{cbs}^\ast(X,W)\to H_{cb}^\ast(X,W)
$$
(we cannot realize $H^\ast(\nu_{G,A}^\ast)$ as the map induced by a morphism
of pairs of resolutions just because we are not able to prove
that the pair $C^*_{cb}(\wdtX)$, $C^*_{cb}(\wdtW)$ provides a pair of resolutions
for $(G,A;\R)$ -- see Remark~\ref{noexact} below).

It readily follows from the definitions that the following diagram commutes:
$$
\xymatrix{
H_{cbs}^*(X,W)\ar[rr]^{H^*(\eta_{G,A}^\ast)}
\ar[rd]_{ H^*(\nu_{G,A}^\ast)} & &
  H_b^*(X,W)\\
& H_{cb}^*(X,W)\ar[ru]_{H^\ast(\rho_b^\ast)}  & \\
}
$$
where $H^*(\rho_b^*)\colon H_{cb}^*(X,W)\to H_b^*(X,W)$ be the map described in the
Introduction.

Let us now suppose that
$(X,W)$ is good. Then Theorem~\ref{isometric:iso} implies that 
the map $H^*(\eta_{G,A}^\ast)$ is an isometric isomorphism,
so the map $H^\ast(\nu_{G,A}^\ast)\circ H^*(\eta_{G,A}^\ast)^{-1}$ provides
a right inverse to $H^*(\rho^*_b)$.
Since  $H^*(\nu_{G,A}^\ast)$ is norm non-increasing,
this map is an isometric embedding, and this concludes the proof of
Theorem~\ref{teo2}.

\begin{rem}\label{noexact}
Suppose that $(X,W)$ is good. If we were able to prove that the complexes $C^\ast_{cb}(\wdtX)$, $C^\ast_{cb}(\wdtW)$ provide a proper pair of resolutions
for $(G,A;\R)$, then we could prove that $H^\ast(\rho_b^\ast)\colon
H^\ast_{cb}(X,W)\to H^\ast_b(X,W)$ is an isometric isomorphism for every
good pair $(X,W)$. However, it is not clear why Ivanov's contracting homotopies
should take continuous cochains into continuous cochains, thus restricting to
contracting homotopies for $C^\ast_{cb}(\wdtX)$, $C^\ast_{cb}(\wdtW)$.
\end{rem}

\subsection{(Unbounded) continuous cohomology of pairs}\label{unbounded:sub}
We conclude the section by proving Theorem~\ref{teo2bis},
which asserts that, when $(X,W)$ is a locally finite good CW-pair,
the map
$$H^*(\rho^\ast)\colon H_{c}^*(X,W)\longrightarrow H^*(X,W)$$
is an isometric isomorphism.

We first observe that, since $W$ is closed in $X$, the subspace
$S_n(W)$ is closed in $S_n(X)$ for every $n\in\matN$. Therefore,
by Tietze's Theorem
every continuous cochain
on $W$ extends to a continuous cochain on $X$, \emph{i.e.}~the restriction
map $C^\ast_c(X)\to C^\ast_c (W)$ is surjective. As a consequence,
both rows of the following commutative diagram are exact:
$$
\xymatrix{
H^{n+1}_c(X)\ar[r] \ar[d]& H^{n+1}_c (W)\ar[r]\ar[d] & H^n_c(X,W)\ar[r]\ar[d]^{H^n(\rho^\ast)} 
& H^n_c(X)\ar[r]\ar[d] & H^n_c(W)\ar[d]\\
H^{n+1}(X)\ar[r] & H^{n+1} (W)\ar[r] & H^n(X,W)\ar[r] & H^n(X)\ar[r] & H^n(W)\ .
} 
$$
Being locally finite, both $X$ and $W$ are metrizable, so
we know from~\cite[Theorem 1.1]{Frigerio} that, in the absolute case,
the vertical arrows are isomorphisms, and the Five Lemma implies
now that $H^n(\rho^\ast)$ is an isomorphism. We are left to show that
it is also an isometry.

The inclusions $C_b^\ast(X,W)\hookrightarrow C^\ast(X,W)$, 
$C_{cb}^\ast(X,W)\hookrightarrow C_c^\ast(X,W)$
induce the \emph{comparison maps}
$c^\ast\colon H^*_b(X,W)\rightarrow H^*(X,W)$, 
$c_c^\ast\colon H^*_{cb}(X,W)\rightarrow H_c^*(X,W)$
and it follows from
the very definitions that for every $\varphi\in H^n(X,W)$, $\varphi_c\in H^n_c(X,W)$
the following equalities hold:
$$
\begin{array}{ccc}
\|\varphi\|_\infty &=&\inf\{\|\psi\|_\infty\;|\;\psi \in H_b^n(X,W),
c^n(\psi)=\varphi\}\ ,\\
\|\varphi_c\|_\infty &=&\inf\{\|\psi_c\|_\infty \;|\;\psi_c \in H_{cb}^n(X,W),
c_c^n(\psi_c)=\varphi_c\}\ ,
\end{array}
$$
where understand that $\inf \emptyset =+\infty$.
Moreover, since $H^\ast(\rho^\ast)\circ c_c^\ast=c^\ast\circ H^\ast(\rho_b^\ast)$, 
for every $\varphi_c\in H^\ast_c(X,W)$ we have
$$
\begin{array}{lll}
\| H^\ast(\rho^\ast)(\varphi_c)\|_\infty &=& 
\inf\{\|\psi\|_\infty \;|\;\psi \in H_b^*(X,W),
c^*(\psi)=H^\ast(\rho^\ast)(\varphi_c)\}\\&=&
\inf\{\|\psi_c\|_\infty \;|\;\psi_c \in H_{cb}^*(X,W),
c^*(H^\ast(\rho_b^\ast)(\psi_c))=H^\ast(\rho^\ast)(\varphi_c)\}\\&=&
\inf\{\|\psi_c\|_\infty \;|\;\psi_c \in H_{cb}^*(X,W),
H^\ast(\rho^\ast)(c_c^\ast(\psi_c))=H^\ast(\rho^\ast)(\varphi_c)\}\\&=&
\inf\{\|\psi_c\|_\infty \;|\;\psi_c \in H_{cb}^*(X,W),
c_c^\ast(\psi_c)=\varphi_c\}\\&=&
\| \varphi_c \|_\infty
\end{array}
$$
where the second equality is due to Theorem~\ref{teo2} (recall that locally finite CW-pair are countable).
The proof of Theorem~\ref{teo2bis}
is now complete.

\section{The duality principle}\label{duality:sec}
This section is mainly devoted to the proof of Theorem~\ref{teo1}.
As already mentioned in the Introduction, once a suitable duality
pairing between measure homology and continuous bounded cohomology is established,
Theorem~\ref{teo1} can be easily deduced from Theorem~\ref{teo2}.

\subsection{Duality between singular homology and bounded cohomology}
Let us begin by recalling the well-known duality between bounded cohomology
and singular homology. Let $(X,W)$ be any pair of topological spaces. 
By definition, $C^n(X,W)$ is the algebraic dual
of $C_n(X,W)$, and it is readily seen that the $L^\infty$-norm
on $C^n(X,W)$ is dual to the $L^1$-norm on $C_n(X,W)$. 
As a consequence, $C_b^n(X,W)$ coincides with the topological dual of
$C_n(X,W)$. This does \emph{not} imply that
$H_b^n(X,W)$ is the dual of $H_n(X,W)$, because taking duals
of normed chain complexes 
does not commute in general with homology (see~\cite{Loh2} for a detailed
discussion of this issue). However, if we denote by
$$
\langle\cdot , \cdot  \rangle\colon H_b^n(X,W)\times H_n(X,W)\to \R
$$
the \emph{Kronecker product} induced by the pairing 
$C_b^n(X,W)\times C_n(X,W)\to \R$, then 
an application 
of Hahn-Banach Theorem (see \emph{e.g.}~\cite[Theorem 3.8]{Lohtesi} for the details)
gives the following:

\begin{prop}\label{duality:standard}
 For every $\alpha\in H_n(X,W)$ we have
$$
\|\alpha\|_1=\sup\left\{ \frac{1}{\|\varphi\|_\infty}\, \big|\, \varphi\in H_b^n(X,W),\ \langle\varphi,
\alpha\rangle =1\right\}\ ,
$$
where we understand that $\sup\emptyset =0$.
\end{prop}

\subsection{Duality between measure homology and continuous bounded cohomology}
The topological dual of $\calC_*(X,W)$ does not admit an easy description,
so in order to compute seminorms in $\calH_\ast (X,W)$ via duality
more work is needed. 
We first observe that, 
if $\mu$ is any measure on $S_n(X)$ with compact determination set and $f$ is any continuous
function on $S_n (X)$, it makes sense to integrate $f$ with respect to $\mu$. Therefore,
for every $n\in\matN$ the bilinear pairing
$$
\langle\cdot,\cdot\rangle\colon
C^n_{cb}(X,W)\times \calC_n(X,W)\to \R,\qquad
\langle f,\mu\rangle=\int_{S_n(X)} f(\sigma)\, d\mu(\sigma)
$$
is well-defined. It readily follows from the definitions
that $|\langle f,\mu\rangle |\leq \| f\|_\infty\cdot
\|\mu\|_\mea$ for every $f \in C^n_{cb}(X,W)$, $\mu\in \calC_n(X,W)$, so
$C^*_{cb}(X,W)$ lies in the topological dual of $\calC_*(X,W)$.
Moreover, for every $i\in\matN$, $f \in C^i_{cb}(X,W)$ and $\mu\in \calC_{i+1}(X,W)$ 
we have $\langle \delta f,\mu\rangle= \langle f, \partial \mu\rangle$,
so this pairing defines a Kronecker product
$$
\langle\cdot,\cdot\rangle\colon H_{cb}^n(X,W)\times 
\mathcal{H}_n(X,W)\to \R
$$
such that 
\begin{equation}\label{duality:eq}
|\langle \varphi_c,\alpha\rangle|\leq \|\varphi_c\|_\infty \cdot \|\alpha\|_\mh\quad 
{\rm for\ every}\ \varphi_c\in H_{cb}^n(X,W),\ \alpha\in \mathcal{H}_n(X,W)\ .
\end{equation}

The following proposition is an immediate consequence 
of inequality~\eqref{duality:eq}, and provides 
a sort of weak duality theorem for continuous bounded cohomology and measure homology.
The term ``weak'' refers to the fact that while Proposition~\ref{duality:standard} 
allows to compute seminorms in homology in terms of seminorms in bounded cohomology, 
here only an inequality is established.
However, this turns out to be sufficient to our purposes. Moreover,
once 
Theorem~\ref{teo1} is proved, one could easily prove
that (in the case of good CW-pairs) the inequality of Proposition~\ref{duality:measure} is in fact an equality,
thus recovering a ``full'' duality between continuous bounded cohomology and measure homology.

\begin{prop}\label{duality:measure}
For every $\alpha \in \mathcal{H}_n(X,W)$ we have
$$\|\alpha\|_\mh\geq \sup\left\{\frac{1}{\|\varphi_c\|_{\infty}}\;\Big{|}\; 
\varphi_c \in H_{cb}^n(X,W),\; \langle\varphi_c,\alpha\rangle=1\right\}\ ,$$
where we understand that $\sup\emptyset=0$.
\end{prop}

\subsection{Proof of the Theorem \ref{teo1}}\label{teo1proof:sub}
We are now ready to conclude the proof of Theorem~\ref{teo1}.
The following result readily follows from the definitions,
and ensures that the 
Kronecker products introduced in the previous subsections
are compatible with each other.

\begin{prop}\label{compatibility}
For every $\varphi_c\in H^n_{cb}(X,W)$,
$\alpha \in H_n(X,W)$ we have
$$\langle H^n(\rho_b^\ast)(\varphi_c),\alpha\rangle=
\langle \varphi_c,H_n(\iota_\ast)(\alpha)\rangle\ .$$
\end{prop}

Let us now suppose that $(X,W)$ is a good CW-pair. We already know
that the map $H_\ast(\iota_\ast)\colon H_\ast(X,W)\to\mathcal{H}_\ast(X,W)$ is a norm non-increasing
isomorphism, so we are left to show that
$\|H_\ast(\iota_\ast)(\alpha)\|_\mh\geq \|\alpha\|_1$ for every $\alpha\in H_\ast(X,W)$.

However, for every $\alpha\in H_n(X,W)$ we have
\begin{align*}
\|H_n(\iota_\ast)(\alpha)\|_\mh &\geq \sup \left\{\frac{1}
{\|\varphi_c\|_\infty}\;\Big{|} \;\varphi_c \in  H_{cb}^n(X,W),
\langle \varphi_c, H_n(\iota_\ast)(\alpha)\rangle=1\right\}\\
&= \sup \left\{\frac{1}{\|\varphi_c\|_\infty}\;\Big{|}\; 
\varphi_c \in  H_{cb}^n(X,W),\langle H^n(\rho_b^\ast)(\varphi_c), \alpha\rangle=1\right\}\\
&
= \sup \left\{\frac{1}{\|\varphi\|_\infty}\;\Big{|}\; 
\varphi \in  H_{b}^n(X,W),\langle \varphi, \alpha\rangle=1\right\}\\
& =\|\alpha\|_1
\ ,
\end{align*}
where the first inequality is due to Proposition~\ref{duality:measure}, the first equality
to Proposition~\ref{compatibility}, the second equality to Theorem~\ref{teo2}, and the 
last equality to Proposition~\ref{duality:standard}.

The proof of Theorem~\ref{teo1} is now complete.

\begin{rem}
Let $(X,W)$ be \emph{any} CW-pair.
 The arguments described in this section show that if 
$H^\ast(\rho_b^\ast)\colon H^*_{cb}(X,W)\to H_b(X,W)$
admits a norm non-increasing right inverse, then the map
$H_*(\iota_*)\colon H_*(X,W)\to \calH_*(X,W)$ is an isometric isomorphism.
\end{rem}

\section{A comparison with Park's seminorms}\label{park:sec}
In \cite{Park}, Park describes
an algebraic foundation of relative bounded cohomology
of pairs, both in the case  of a pair of groups $(G,A)$ equipped with a
homomorphism $A\rightarrow G$
and in the case of a pair of path connected topological spaces $(X,W)$ equipped with a continuous
map $W\rightarrow X$. 
However, recall from the Introduction that the seminorms considered by Park are
quite different from the seminorms considered in this paper, which date back to Gromov~\cite{Gromov}.
In this section we investigate the relationships between our seminorms and the seminorms
introduced in~\cite{Park}, proving in particular that there exist examples 
for which they are \emph{not}
isometric to each other.

\subsection{Park's mapping cone for homology}
Let $(X,W)$ be a countable CW-pair, where both $X$ ad $W$ are connected, 
and let us suppose that the inclusion $i\colon W\hookrightarrow X$
induces an injective map on the fundamental groups
(several considerations here below also hold without this last assumption,
but this is not relevant to our purposes).
We also denote by 
$i_\ast\colon C_*(W)\to C_*(X)$ the map induced by the inclusion $i$.
The homology mapping cone complex of $(X,W)$ 
is the complex $(C_\ast(W\to X),\overline{d}_\ast)=(C_\ast(X)\oplus C_{\ast-1}(W), 
\overline{d}_\ast)$, where
$$
\begin{array}{cccccccc}
\overline{d}_n\colon &C_n(X)&\oplus &C_{n-1}(W) &\longrightarrow & C_{n-1}(X)&\oplus &C_{n-2}(W)\\
&(u_n&,&v_{n-1})&\longmapsto &
(d_n u_n+i_{n-1}(v_{n-1})&,&-d_{n-1} v_{n-1})\ ,
\end{array}
$$
and $d_\ast$ denotes the usual differential both of $C_\ast(X)$ and of $C_\ast (W)$.
The homology of the mapping cone $(C_\ast(W\to X), \overline{d}_\ast)$ 
is  denoted by $H_\ast(W\rightarrow X)$. 
For every $\omega\in[0,\infty)$ one can endow $C_*(W\to X)$ with the 
 $L^1$-norm
$$\|(u,v)\|_1(\omega)=\|u\|_1+ (1+\omega) \|v\|_1\ , $$
which induces in turn a seminorm (still denoted by $\|\cdot \|_1 (\omega)$) 
on $H_*(W\rightarrow X)$ (in fact, in~\cite{Park2} the case $\omega=\infty$
is also considered, but this is not relevant to our purposes).

As observed in~\cite{Park2}, the chain map
\begin{equation}\label{betaPark:eq}
\beta_\ast\colon C_\ast(W\to X)\to C_\ast(X,W)=C_\ast(X)/C_\ast(W),\qquad
\beta_\ast (u,v)=[u]
\end{equation}
induces an isomorphism $$H_\ast(\beta_\ast)\colon H_\ast(W\to X)\to H_\ast(X,W)\ .$$
The explicit description of $\beta_*$ implies that 
$$
\| H_\ast(\beta_\ast)(\alpha) \|_1 \leq \|\alpha \|_1 (0)\leq \|\alpha \|_1 (\omega)
$$
for every 
$\alpha\in H_\ast(W\rightarrow X)$, $\omega\in [0,\infty)$.

\subsection{Park's mapping cone for bounded cohomology}
The mapping cone for bounded cohomology can be defined as the (topological) dual of the mapping cone for homology.
More precisely, let us fix $\omega\in [0,\infty)$, and let us endow $C_\ast(W\to X)$ with
the norm $\| \cdot \|_1(\omega)$. Then it is readily seen that the topological dual of
$C_n(W\to X)=C_n(X)\oplus C_{n-1}(W)$ is isometrically isomorphic to the space
$$
C_b^n (W\to X)=C_b^n(X)\oplus C_b^{n-1}(W)
$$
endowed with the $L^\infty$-norm $\|\cdot\|_\infty (\omega)$ defined by
$$
\|(f,g)\|_\infty(\omega)=\max\{\|f\|_\infty ,(1+\omega)^{-1}
\|g\|_\infty\}\ .
$$
In other words,  
the pairing
$$
C_b^\ast(W\rightarrow X)\times C_\ast(W\rightarrow X)
\to\matR,\qquad 
((f,f'),(a,a'))\mapsto f(a)-f'(a')
$$
realizes $C_b^\ast(W\to X)$ as the  topological dual of $C_\ast(W\to X)$,
and an easy computation shows that the norm $\| \cdot \|_\infty(\omega)$ just introduced
on $C_b^\ast(W\to X)$ coincides with the operator norm (with respect
to the norm $\|\cdot \|_1(\omega)$ fixed on $C_\ast(W\to X)$).
Therefore, 
if $i^\ast\colon C_b^\ast(X)\rightarrow C_b^\ast(W)$ is the 
cochain map induced by the inclusion, then
the cohomology mapping cone complex of $(X,W)$ is 
the complex $(C_b^\ast (W\to X),\overline{\delta}^\ast)$,
where
$\overline{\delta}^\ast$ is defined as the dual map of $\overline{d}_\ast$, and admits therefore the following explicit
description (see~\cite{Park} for the details):
$$
\begin{array}{cccccccc}
\overline{\delta}^n\colon&C_b^n(X)&\oplus &C_b^{n-1}(W) &\longrightarrow 
& C_b^{n+1}(X)&\oplus &C_b^{n}(W)\\
&(f_n&,&g_{n-1})&\longmapsto &(\delta^n f_n&,
&-i^n(f_n)-\delta^{n-1} g_{n-1})\ 
\end{array}
$$
(here $\delta^*$ denotes the usual differential both of $C_b^\ast(X)$ and of $C_b^\ast(W)$).
The cohomology of the complex $(C_b^\ast(W\to X),\overline{\delta^\ast})$ is denoted by
$H_b^\ast(W\rightarrow X)$. 
Just as in the case of homology, 
the $L^\infty$-norm $\|\cdot \|_\infty(\omega)$ on $C^n_b(W\to X)$ 
descends to a seminorm (still denoted by
$\|\cdot \|_\infty(\omega)$)  on $H^\ast_b(W\to X)$.

The chain map
$$
\beta^\ast\colon C^\ast_b(X,W)\to C^\ast_b(W\to X),\qquad
\beta^\ast(f)=(f,0)
$$
is the dual of the chain map $\beta_\ast$ introduced in Equation~\eqref{betaPark:eq}
above, and induces an isomorphism
$$
H^\ast(\beta^\ast)\colon H^\ast_b(X,W)\to H^\ast_b(W\to X)\ 
$$
such that
$$
\| H^\ast(\beta^\ast)(\varphi) \|_\infty(\omega) \leq \|H^\ast(\beta^\ast)(\varphi) \|_\infty (0)\leq 
\|\varphi\|_\infty
$$ 
for every $\varphi\in H^\ast(X,W)$, $\omega\in [0,\infty)$.
More precisely, 
the following result is proved in~\cite[Theorem 4.6]{Park}:

\begin{teo}\label{biLipschitz}
For every $n\in\matN$, the isomorphism
$H^n(\beta^\ast)$
is such that
$$
\frac{1}{n+2}\|\varphi\|_\infty\leq \|H^n(\beta^\ast)(\varphi)\|_\infty(0)
\leq \|\varphi\|_\infty \qquad \text{for\ every}\;\; \varphi\in H^n_b(X,W)\ .
$$
\end{teo}

It is asked in~\cite{Park} 
whether $H^\ast(\beta^\ast)$ is actually an isometry or not.
We show in Proposition~\ref{nonisometry} below
that there exists examples 
for which  $H^\ast(\beta^\ast)$ is \emph{not} an isometry.

\subsection{Mapping cones and duality}
In the previous subsection we have seen that, for every 
$\omega\geq 0$, the normed space $(C^\ast_b(W\to X), \|\cdot \|_\infty(\omega))$ coincides 
with the topological dual of the normed space 
$(C_\ast(W\to X), \|\cdot \|_1(\omega))$. We may therefore apply the duality
result proved in~\cite[Theorem 3.14]{Lohtesi}, and obtain the following:

\begin{prop}\label{dualitycone}
 If the map
$$
H^\ast(\beta^\ast)\colon \left(H^\ast_b(X,W),\|\cdot\|_\infty\right)\to \left(H^\ast_b(W\to X), \|\cdot \|_\infty(\omega)\right)
$$
is an isometric isomorphism, then
$$
\| H_\ast(\beta_\ast)(\alpha)\|_1=\| \alpha \|_1(\omega)
$$
for every $\alpha\in H_\ast(X,W)$.
\end{prop}

\subsection{An explicit example}
Let $M$ be a compact, connected, orientable manifold with connected
boundary,
and suppose that the inclusion
$i\colon \partial M\to M$ induces an injective homomorphism
$i_\ast\colon \pi_1(\partial M)\to\pi_1(M)$.

We denote by $[M,\partial M]$ the (real) fundamental class in 
$H_n(M,\partial M)$ and we set
 $$[\partial M\hookrightarrow M]=H_n(\beta_\ast)^{-1}([M,\partial M])\in H_n(\partial M\to M)\ .$$ 
The $L^1$-seminorm $\| [M,\partial M]\|_1$ of the real fundamental class of $M$ is usually
known as the \emph{simplicial volume} of $M$, and it is denoted simply
by $\| M\|$. Similarly, the $L^1$-seminorm of the real fundamental class
$[\partial M]\in H_{n-1}(\partial M)$ is the simplicial volume of $\partial M$,
and it is denoted by $\|\partial M\|$.

\begin{lemma}\label{conto}
 We have
$$
\| [\partial M\to M ]\|_1 (\omega)\geq \|M\| +(1+\omega)\|\partial M\| \ .
$$
\end{lemma}
\begin{proof}
It is shown in~\cite{Park2} that, if $\alpha\in C_i(M)$ is such that
$d_i \alpha\in C_{i-1}(\partial M)$ (so that $\alpha$ defines an element
$[\alpha]\in H_i(M,\partial M)$), then 
$$
H_i(\beta_\ast)^{-1}([\alpha])=[(\alpha,-d_i \alpha)]\ .
$$
Therefore, if $\alpha\in C_n(M)$ is a representative of
the fundamental class $[M,\partial M]\in H_n(M,\partial M)$, then
$(\alpha,-d_n \alpha)$ is a representative of $[\partial M\hookrightarrow M]\in H_n(\partial M\to M)$.
If $(\alpha',\gamma)$ is any other representative of such a class,
then by definition of mapping cone there exist $x\in C_{n+1}(M)$ and $y\in C_n(\partial M)$ such that:
\begin{equation*}
\left \{
\begin{array}{lcc}  
\alpha-\alpha'&=&d_{n+1} x+i_n (y)\\
\gamma+d_n \alpha&=&- d_{n} y\ .
\end{array}
\right.
\end{equation*}
These equalities readily imply that $[\alpha']=[\alpha]$ in $H_n(M,\partial M)$
and 
$[\gamma]=[-d_n \alpha]$ in $H_{n-1}(\partial M)$.
As a consequence, since $d_n \alpha$ is a representative
of the fundamental class of $\partial M$, we have
$\| \alpha '\|_1\geq \|[\alpha']\|_1=\| M\|$ and $\|\gamma \|_1\geq
\|[\gamma]\|_1=\| \partial M\|$, whence 
$$
\| (\alpha',\gamma)\|_1 (\omega)\geq \|M\| +(1+\omega)\|\partial M\|\ .
$$
The conclusion follows from the fact that
$(\alpha',\gamma)$ is an arbitrary representative of $[\partial M\to M]$.
\end{proof}

\begin{prop}\label{nonisometry}
 Let $M$ be a compact connected orientable hyperbolic $n$-manifold 
with connected geodesic boundary.
Then, for every $\omega\in [0,\infty)$ the isomorphism
$$
H^n(\beta^\ast)\colon \left(H^n_b(X,W),\|\cdot \|_\infty\right)\to 
\left(H_b^n(W\hookrightarrow X),\|\cdot \|_\infty(\omega)\right)
$$
is not isometric.
\end{prop}
\begin{proof}
It is well-known that the inclusion $\partial M\hookrightarrow M$ induces an injective
map on fundamental groups. Moreover, since $\partial M$ is a closed
orientable hyperbolic $(n-1)$-manifold, we also have $\|\partial M\|>0$.
By Proposition~\ref{dualitycone}, if $H^n(\beta^\ast)$ were an isometry we would have
$\|[\partial M\to M]\|_1(\omega)=\|[M,\partial M]\|_1=\| M\|$, and this contradicts
Lemma~\ref{conto}.
\end{proof}

\bibliographystyle{amsalpha}
\bibliography{biblio}

\providecommand{\bysame}{\leavevmode\hbox to3em{\hrulefill}\thinspace}
\providecommand{\MR}{\relax\ifhmode\unskip\space\fi MR }
\providecommand{\MRhref}[2]{%
  \href{http://www.ams.org/mathscinet-getitem?mr=#1}{#2}
}
\providecommand{\href}[2]{#2}
\begin{thebibliography}{Mon01}

\bibitem[BE78]{Bieri}
R.~Bieri and B.~Eckmann, \emph{Relative homology and {P}oincar\'e duality for
  group pairs}, J. Pure Appl. Algebra \textbf{13} (1978), 277--319.

\bibitem[Ber08]{Berlanga}
R.~Berlanga, \emph{A topologised measure homology}, Glasg. Math. J. \textbf{50}
  (2008), 359--369.

\bibitem[BH99]{Bridson}
M.~Bridson and A.~Haefliger, \emph{Metric spaces of non-positive curvature},
  Grundlehren der Mathematischen Wissenschaften [Fundamental Principles of
  Mathematical Sciences], no. 319, Springer-Verlag, Berlin, 1999.

\bibitem[Bou52]{para2}
D.~G. Bourgin, \emph{The paracompactness of the weak simplicial complex}, Proc.
  Nat. Acad. Sci. U.S.A. \textbf{38} (1952), 305--313.

\bibitem[FM]{Fujiwara}
K.~Fujiwara and J.~F. Manning, \emph{Simplicial volume and fillings of
  hyperbolic manifolds}, arXiv:1012.1039.

\bibitem[FP10]{FriPag}
R.~Frigerio and C.~Pagliantini, \emph{The simplicial volume of hyperbolic
  manifolds with geodesic boundary}, Algebr. Geom. Topol. \textbf{10} (2010),
  979--1001.

\bibitem[Fri11]{Frigerio}
R.~Frigerio, \emph{({B}ounded) continuous cohomology and {G}romov's
  proportionality principle}, Manuscripta Math. \textbf{134} (2011), 435--474.

\bibitem[Gro82]{Gromov}
M.~Gromov, \emph{Volume and bounded cohomology}, Inst. Hautes \'Etudes Sci.
  Publ. Math. \textbf{56} (1982), 5--99.

\bibitem[Han98]{Hansen}
S.K. Hansen, \emph{Measure homology}, Math. Scand. \textbf{83} (1998),
  205--219.

\bibitem[Iva87]{Ivanov}
N.V. Ivanov, \emph{Foundations of the theory of bounded cohomology}, J. Soviet
  Math. \textbf{37} (1987), 1090--1114.

\bibitem[L{\"o}h06]{Loh}
C.~L{\"o}h, \emph{Measure homology and singular homology are isometrically
  isomorphic}, Math. Z. \textbf{253} (2006), 197--218.

\bibitem[L{\"o}h07]{Lohtesi}
\bysame, \emph{$l^1$--homology and simplicial volume}, Ph.D. thesis, WWU
  M{\"u}nster, 2007, available online at
  http://nbn-resolving.de/urn:nbn:de:hbz:6-37549578216.

\bibitem[L{\"o}h08]{Loh2}
\bysame, \emph{Isomorphisms in $l^1$-homology}, M\"unster J. Math. \textbf{1}
  (2008), 237--266.

\bibitem[LS09]{Loh-Sauer}
C.~L{\"o}h and R.~Sauer, \emph{Degree theorems and {L}ipschitz simplicial
  volume for non-positively curved manifolds of finite volume}, J. Topology
  \textbf{2} (2009), 193--225.

\bibitem[Miy52]{para}
H.~Miyazaki, \emph{The paracompactness of ${CW}$--complexes}, Tohoku Math. J.
  \textbf{4} (1952), 309--313.

\bibitem[Mon01]{Monod}
N.~Monod, \emph{Continuous bounded cohomology of locally compact groups},
  Lecture notes in Mathematics, no. 1758, Springer-Verlag, Berlin, 2001.

\bibitem[MY07]{Yaman}
I.~Mineyev and A.~Yaman, \emph{Relative hyperbolicity and bounded cohomology},
  available online at http://www.math.uiuc.edu/~mineyev/math/art/rel-hyp.pdf,
  2007.

\bibitem[Pag11]{Pagliantini}
C.~Pagliantini, Ph.D. thesis, University of Pisa, 2011, in preparation.

\bibitem[Par03]{Park}
H.~S. Park, \emph{Relative bounded cohomology}, Topology Appl. \textbf{131}
  (2003), 203--234.

\bibitem[Par04]{Park2}
\bysame, \emph{Foundations of the theory of $l^1$-homology}, J. Korean Math.
  Soc. \textbf{41} (2004), 591--615.

\bibitem[Thu79]{Thurston}
W.~P. Thurston, \emph{The geometry and topology of $3$-manifolds}, Princeton,
  1979, mimeo\-graphed notes.

\bibitem[Zas98]{Zastrow}
A.~Zastrow, \emph{On the (non)-coincidence of {M}ilnor-{T}hurston homology
  theory with singular homology theory}, Pacific J. Math. (1998), 369--396.

\end{thebibliography}

\end{document}